\documentclass[11pt,reqno]{amsart}

\usepackage{amsmath,amssymb,mathrsfs,amsthm,mathabx}
\usepackage{graphicx,cite,cases}
\usepackage{xcolor}
\newtheorem{theorem}{Theorem}[section]

\newtheorem{lemma}[theorem]{Lemma}

\numberwithin{equation}{section}

\begin{document}

\title[elastic wave scattering by biperiodic structures]{An
adaptive finite element DtN method for the elastic wave scattering by 
biperiodic structures}

\author{Gang Bao}
\address{School of Mathematical Sciences, Zhejiang University, Hangzhou 310027,
China.}
\email{baog@zju.edu.cn}

\author{Xue Jiang}
\address{Faculty of Science, Beijing University of Technology, Beijing,
100124, China.}
\email{jxue@lsec.cc.ac.cn}

\author{Peijun Li}
\address{Department of Mathematics, Purdue University, West Lafayette, Indiana
47907, USA.}
\email{lipeijun@math.purdue.edu}

\author{Xiaokai Yuan}
\address{School of Mathematical Sciences, Zhejiang University, Hangzhou 310027,
China.}
\email{yuan170@zju.edu.cn}

\thanks{The work of GB is supported in part by an NSFC Innovative Group Fund
(No.11621101). The work of XJ is supported in part by China NSF grants 11771057
and 11671052. The research of PL is supported partially by the NSF grant
DMS-1912704.}



\begin{abstract}
Consider the scattering of a time-harmonic elastic plane wave by a bi-periodic
rigid surface. The displacement of elastic wave motion is modeled by the
three-dimensional Navier equation in an open domain above the surface. Based on
the Dirichlet-to-Neumann (DtN) operator, which is given as an infinite series,
an exact transparent boundary condition is introduced and the scattering problem
is formulated equivalently into a boundary value problem in a bounded domain. An
a posteriori error estimate based adaptive finite element DtN method is proposed
to solve the discrete variational problem where the DtN operator is truncated
into a finite number of terms. The a posteriori error estimate takes account of
the finite element approximation error and the truncation error of the DtN
operator which is shown to decay exponentially with respect to the truncation
parameter. Numerical experiments are presented to illustrate the effectiveness
of the proposed method.    
\end{abstract}

\maketitle

\section{Introduction}\label{section:introduction}

This paper concerns the scattering of a time-harmonic elastic plane wave by
a bi-periodic surface in three dimensions. Due to the wide and significant
applications in seismology and geophysics, the elastic wave scattering problems
have received ever increasing attention in both mathematical and engineering
communities \cite{arens99, arens99integral, lwz-ip15}. Compared
with the acoustic and electromagnetic wave scattering problems, the elastic wave
scattering problems are less studied due to the fact that the elastic wave
consists of coupled compressional and shear wave components with different
wavenumbers, which makes the analysis of the problems more complicated. In
addition, there are two challenges for the elastic surface scattering problem:
the solution may have singularity due to a possible non-smooth surface; the
problem is imposed in an open domain. In this paper, we intend to address
both of these two issues by proposing an a posteriori error estimate based
adaptive finite element method with the transparent boundary condition. 

The a posteriori error estimates are computable quantities from numerical
solutions. They can be used to measure the solution errors of discrete problems
without requiring any a priori information of exact solutions
\cite{BR-SJNA-1978, M-JCAM-1998}. Since the a posteriori error estimate based
adaptive finite element method has the ability to control the error and to
asymptotically optimize the approximation, it is crucial for mesh modification
such as refinement and coarsening \cite{D-SJNA-1996, V-1996}. The method has
become an important numerical tool for solving boundary value problems of
partial differential equations, especially for those where the solutions have
singularity or multiscale phenomena.

The key of overcoming the second issue is to reformulate the open domain
problem into a boundary value problem in a bounded domain without generating
artificial wave reflection. One possible approach is to make use of the
perfectly matched layer (PML) techniques. The basic idea of the PML is to
surround the domain of interest by a layer of finite thickness of fictitious
medium that may attenuate the waves propagating from inside of the computational
domain. When the waves reach the outer boundary of the PML region, their
amplitudes are so small that the homogeneous Dirichlet boundary
condition can be imposed. Due to the effectiveness and simplicity, since
B\'erenger  proposed the technique to solve the time domain Maxwell equations
\cite{b-jcp94}, it has undergone a tremendous development of designing various
PML methods to solving a wide range of open domain scattering problems
\cite{BW-SJNA-2005, BP-MC-07, bpt-mc10, CW-MOTL-1994, CM-SISC-1998, ct-g01,
HSB-jasa96, HSZ-sima03, LS-C-1998}. Combined with adaptive finite element
methods, the PML method has been investigated to solve the two- and
three-dimensional obstacle scattering problems \cite{CC-mc08, CL-sinum05,
cxz-mc16} and the two- and three-dimensional diffraction grating problems
\cite{BPW-MC-2010, CW-SINUM-2003, jllz-m2na17}. The a posteriori error estimates
based adaptive finite element PML methods take account of the finite element
discretization errors and the PML truncation errors which decay exponentially
with respect to the PML parameters. 

Alternatively, another effective approach to truncate the open domain is to
construct the Dirichlet-to-Neumann (DtN) map and introduce the transparent
boundary condition to enclose the domain of interest. Since the DtN
operator is exact, the transparent boundary condition can be imposed on the
boundary which is chosen as close as possible to the scattering structure.
Compared to the PML method, the DtN method can reduce the size of the
computational domain. As a viable alternative to the PML method, the adaptive
finite element DtN methods have also been developed recently to solve many two-
and three-dimensional scattering problems, such as the acoustic scattering
problems \cite{JLLZ-JSC-2017, JLZ-CCP-2013, WBLLW-2015-SJNA}, the
three-dimensional electromagnetic scattering problem \cite{JLLWWZ}, and
the two-dimensional elastic wave scattering problems
\cite{LY-2019-periodic, LY-2019-obstacle}.

This paper concerns the numerical solution of the elastic wave
scattering by biperiodic structures in three dimensions. It is a non-trivial
extension of the elastic wave scattering by periodic structures in two
dimensions \cite{LY-2019-periodic}. There are two challenges for the
three-dimensional problem. First, the Helmholtz decomposition of the elastic
wave equation gives two two-dimensional Helmholtz equations in the
two-dimensional case; however, for the Helmholtz decomposition in the
three-dimensional case, we have to consider a three-dimensional Helmholtz
equation and a three-dimensional Maxwell equation, which makes the analysis much
more complicated. Second, from the computational point of view, it is much more 
time-consuming to solve the three-dimensional problem than to solve the
two-dimensional problem. 

Specifically, we consider the scattering of a time-harmonic plane elastic wave
by a biperiodic rigid surface. The elastic wave propagation is modeled by the
three-dimensional Navier equation in the open domain above the scattering
surface. By the Helmholtz decomposition, a DtN operator is constructed in terms
of Fourier series expansions for the compressional and shear wave components,
then an exact transparent boundary condition is introduced to reduce the open
domain problem into an equivalent boundary value problem in a bounded domain.
The nonlocal DtN operator needs to be truncated into a sum of finitely many
terms in actual computation. However, it is known that the convergence of the
truncated DtN operator could be arbitrary slow to the original DtN operator in
the operator norm \cite{HNPX-JCAM-2011}. By carefully examining the properties
of the exact solution, we observe that the truncated DtN operator converges
exponentially to the original DtN operator when acting on the solution of the
elastic wave equation, which enables the analysis of exponential convergence
for this work. Combined with the truncated DtN operator and finite element
method, the discrete problem is studied. We develop a new duality argument to
deduce the a posteriori error estimate. The a posteriori error estimate takes
account of the finite element approximation error and the DtN operator
truncation error which is shown to decay exponentially with respect to the
truncation parameter $N$. Moreover, an a posteriori error estimate based
adaptive finite element algorithm
is presented to solve the discrete problem, where the estimate is used to design
the algorithm to choose elements for refinements and to determine the truncation
parameter $N$. Due to the exponential convergence of the truncated DtN operator,
the choice of the truncation parameter $N$ turns out not to be sensitive to the
given tolerance of accuracy. Numerical examples are presented to demonstrate the
effectiveness of the proposed method.

The paper is organized as follows. In Section \ref{section: problem}, the model
equation is introduced for the scattering problem.  Section \ref{section: BVP}
concerns the variational problem. By the Helmholtz decomposition, the DtN
operator is constructed and the transparent boundary condition is introduced to
reformulate the scattering problem into a boundary value problem in a bounded
domain, and the corresponding weak formulation is presented. In Section
\ref{section: DP}, the discrete problem is studied by using the finite element
method with the truncated DtN operator. Section \ref{section: PEA} is devoted to
the a posteriori error analysis for the discrete problem and the exponential
convergence is proved for the truncated DtN operator. In Section
\ref{section: NE}, an adaptive finite element algorithm is described and
numerical experiments are carried out to illustrate the competitive behavior of
the proposed method. The paper is concluded with some general remarks and
directions for future work in Section \ref{section: conclusion}.

\section{Problem formulation}\label{section: problem}

Consider the scattering of a time-harmonic plane elastic wave by a rigid
biperiodic surface. Due to the biperiodic structure, the scattering problem can
be restricted into a single biperiodic cell, as shown in Figure \ref{pg}. Let
\[
 S=\{\boldsymbol{x}=(x_1, x_2, x_3)^\top\in\mathbb{R}^3:
(x_1, x_2)\in(0, \Lambda_1)\times(0, \Lambda_2), x_3=f(x_1, x_2)\}
\]
be the scattering surface, where $f$ is a Lipschitz continuous biperiodic
function with periods $\Lambda_1$ and $\Lambda_2$ in the $x_1$ and $x_2$
directions, respectively. Denote the open space above $S$ by 
\[
 \Omega_f=\{\boldsymbol{x}\in\mathbb{R}^3:
(x_1, x_2)\in(0, \Lambda_1)\times(0, \Lambda_2), x_3>f(x_1, x_2)\},
\]
which is assumed to be filled with an isotropic and homogeneous elastic
medium. The medium may be characterized by the Lam\'{e} parameters $\lambda,
\mu$ and the mass density $\rho$ which is assumed to be unit for simplicity.  
Furthermore, we assume that the Lam\'{e} constants satisfy $\mu>0,
\lambda+\mu>0$. Define 
\[
\Gamma_h=\{\boldsymbol{x}\in\mathbb{R}^3: (x_1, x_2)\in(0, \Lambda_1)\times(0,
\Lambda_2), x_3= h\},
\]
where $h$ is a constant satisfying $h>\max_{(x_1, x_2)\in(0, \Lambda_1)\times(0,
\Lambda_2)}f(x_1, x_2)$. Denote by $\Omega$ the bounded domain enclosed by $S$
and $\Gamma_h$, i.e., 
\[
\Omega=\{\boldsymbol{x}\in\mathbb{R}^3: (x_1, x_2)\in(0, \Lambda_1)\times
(0, \Lambda_2), f(x_1, x_2)< x_3 < h\}. 
\]

\begin{figure}
\centering
\includegraphics[width=0.5\textwidth]{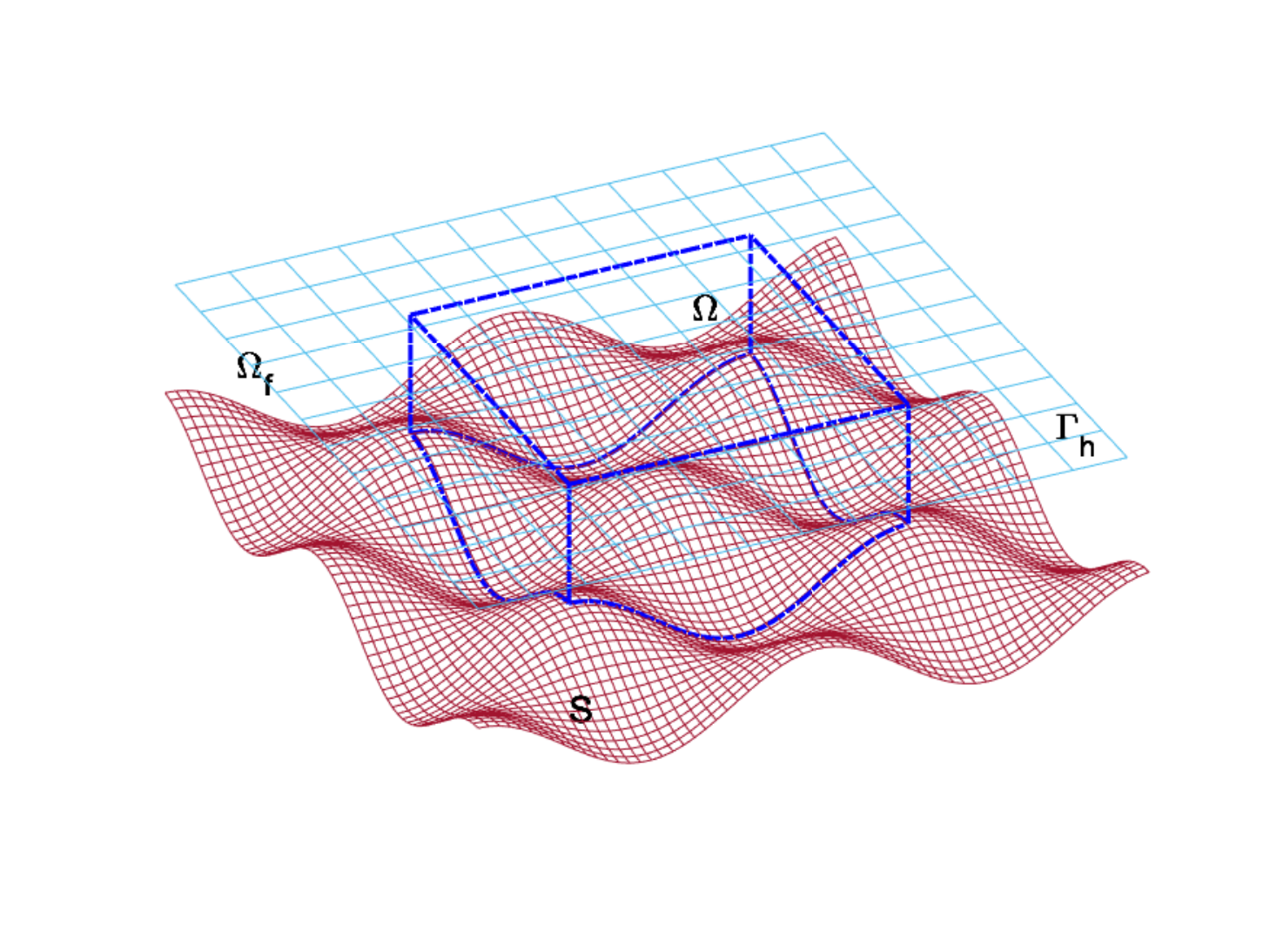}
\caption{Problem geometry of the elastic scattering by a biperiodic surface.}
\label{pg}
\end{figure}

Let a compressional plane wave
\[
\boldsymbol{u}^{\rm inc}(\boldsymbol x)=\boldsymbol{q} e^{{\rm i}\kappa_1
\boldsymbol{q}\cdot\boldsymbol{x}}
\]
be sent from the above to impinge the surface,  where
$\boldsymbol{q}=(\sin\theta_1 \cos\theta_2, \sin\theta_1\sin\theta_2,
-\cos\theta_1)^\top$, $\theta_1\in[0, \pi/2)$ and $\theta_2\in[0, 2\pi]$ 
are the incident angles, and $\kappa_1=\omega/\sqrt{\lambda+2\mu}$ is the
compressional wavenumber with $\omega$ being the angular frequency. We mention
that the results are the same for the
incidence of a shear plane wave $\boldsymbol{u}^{\rm
inc}(\boldsymbol{x})=\boldsymbol{p} e^{{\rm i}\kappa_2 \boldsymbol
q\cdot\boldsymbol{x}}$, where $\boldsymbol p$ is a unit vector satisfying
$\boldsymbol p\cdot\boldsymbol q=0$ and $\kappa_2=\omega/\sqrt{\mu}$ is the
shear wavenumber, or a linear combination of the shear and compressional plane
waves. 

Denote the displacement of the scattered wave by $\boldsymbol{u}$, which
satisfies the elastic wave equation
\begin{equation}\label{ewe}
\mu\Delta\boldsymbol{u}+(\lambda+\mu)\nabla\nabla\cdot\boldsymbol{u}
+\omega^2\boldsymbol{u}=0 \quad {\rm in} ~ \Omega_f. 
\end{equation}
Since the surface is assumed to be elastically rigid, we have 
\begin{equation}\label{bc}
\boldsymbol{u}=-\boldsymbol{u}^{\rm inc} \quad {\rm on} ~ S.
\end{equation}
In addition, the scattered wave $\boldsymbol u$ is assumed to satisfy the
bounded outgoing wave condition as $x_3\to\infty$. Motivated by uniqueness, we
seek the so-called quasi-periodic solutions of \eqref{ewe}--\eqref{bc}, i.e.,
$\boldsymbol{u}(\boldsymbol x)e^{-{\rm
i}\boldsymbol{\alpha}\cdot\boldsymbol{r}}$ is a biperiodic function of
$\boldsymbol r=(x_1, x_2)^\top$ with periods $\Lambda_1$ and $\Lambda_2$ in the
$x_1$ and $x_2$ directions, respectively, where $\boldsymbol\alpha=(\alpha_1,
\alpha_2)^\top, \alpha_1=\kappa_1\sin\theta_1\cos\theta_2,
\alpha_2=\kappa_1\sin\theta_1 \sin\theta_2$. 

Define a quasi-periodic function space
\begin{eqnarray*}
 H_{\rm qp}^1(\Omega)=\{u\in H^1(\Omega) & : & e^{{\rm i}\alpha_1\Lambda_1}u(0,
x_2, x_3)=u(\Lambda_1, x_2, x_3),\\
&& e^{{\rm i}\alpha_2\Lambda_2}u(x_1, 0, x_3)=u(x_1, \Lambda_2, x_3)\}
\end{eqnarray*}
and its subspace 
\[
 H_{S, {\rm qp}}^1(\Omega)=\{u\in H_{\rm qp}^1(\Omega): u=0\text{ on } S\}. 
\]
Let 
\begin{eqnarray*}
L_{\rm qp}^2(\Gamma_h)=\{u\in L^2(\Gamma_h) &:&e^{{\rm i}\alpha_1\Lambda_1}u(0,
x_2, h)=u(\Lambda_1, x_2, h),\\
&&e^{{\rm i}\alpha_2\Lambda_2}u(x_1, 0, h)=u(x_1, \Lambda_2, h)\}. 
\end{eqnarray*}
For any $u\in L^2_{\rm qp}(\Gamma_h)$,  it has the Fourier series expansion 
\[
u(\boldsymbol{r}, h)=\sum\limits_{n\in\mathbb{Z}^2}u_n(h)e^{{\rm
i}\boldsymbol{\alpha}_n\cdot\boldsymbol{r}},\quad
u_n(h)=\frac{1}{\Lambda_1
\Lambda_2}\int_{0}^{\Lambda_1}\int_{0}^{\Lambda_2} u(\boldsymbol{r}, h)e^{-{\rm
i}\boldsymbol{\alpha}_n\cdot\boldsymbol{r}}{\rm d}\boldsymbol{r},
\]
where $n=(n_1, n_2)^\top\in\mathbb Z^2$,
$\boldsymbol{\alpha}_n=(\alpha_{1 n}, \alpha_{2 n})^\top,
\alpha_{j n}=\alpha_j+2\pi n_j/\Lambda_j, j=1, 2$. 

Define a trace function space $H^s(\Gamma_{h}), s\in\mathbb R^+$ by
\[
H^{s}(\Gamma_{h})=\left\{u\in L^2(\Gamma_{h}):
\|u\|_{H^{s}(\Gamma_{h})}<\infty\right\},
\]
where the norm is given by
\[
\|u\|^2_{H^{s}(\Gamma_{h})}=\Lambda_1\Lambda_2
\sum\limits_{n\in\mathbb{Z}^2}\left(1+|\boldsymbol{\alpha}_n|^2\right)^s 
\left|u_n(h)\right|^2.
\]
It is clear that the dual space of $H^{s}(\Gamma_h)$ is $H^{-s}(\Gamma_h)$ with
respect to the scalar product in $L^2(\Gamma_h)$ given by
\[
\langle u, v\rangle_{\Gamma_h}=\int_{\Gamma_h} u\bar{v} {\rm d}s.
\]

Let $\boldsymbol{H}_{\rm qp}^{1}(\Omega), \boldsymbol{H}_{S,
{\rm qp}}^{1}(\Omega) $ and $\boldsymbol{H}^{s}(\Gamma_{h})$ be the 
Cartesian product spaces equipped with the corresponding 2-norms of
$H_{\rm qp}^{1}(\Omega), H_{S, {\rm qp}}^{1}(\Omega)$ and 
$H^{s}(\Gamma_{h})$, respectively. Throughout the paper, the notation $a\lesssim
b$ stands for $a\leq Cb$, where $C$ is a positive constant whose value is not
required but should be clear from the context. 

\section{The boundary value problem}\label{section: BVP}

In this section, we introduce the DtN operator to reduce the problem
\eqref{ewe}--\eqref{bc} into a boundary value problem in the bounded domain
$\Omega$ and present the well-posedness of its variational formulation. 

Consider the Helmholtz decomposition 
\begin{equation}\label{Helmholtz}
\boldsymbol{u}=\nabla\phi+\nabla\times\boldsymbol{\psi},\quad
\nabla\cdot\boldsymbol{\psi}=0\quad {\rm in} ~ \Omega,
\end{equation}
where $\phi$ is a scalar potential function and $\boldsymbol\psi=(\psi_1,
\psi_2, \psi_3)$ is a vector potential function. Substituting \eqref{Helmholtz}
into the elastic wave equation \eqref{ewe}, we may verify that $\phi$ and
$\boldsymbol{\psi}$ satisfy the following Helmholtz equation and the Maxwell
equation, respectively:
\begin{equation}\label{Helmholtz1}
\Delta\phi+\kappa_1^2\phi=0,\qquad \nabla\times\left(\nabla\times
\boldsymbol{\psi}\right)-\kappa_2^2\boldsymbol{\psi}=0.
\end{equation}
It is also easy to verify from the Helmholtz decomposition \eqref{Helmholtz}
and the boundary condition \eqref{bc} that $\phi$ and $\boldsymbol\psi$ satisfy
the following coupled boundary conditions on $S$:  
\begin{equation}\label{Helmholtz2}
\partial_{\nu}\phi+\left(\nabla\times\boldsymbol{\psi}\right)\cdot
\nu=-\boldsymbol{u}^{\rm inc}\cdot\nu,\quad
\left(\nabla\times\boldsymbol{\psi}
\right)\times\nu+\nabla\phi\times\nu=-\boldsymbol{u}^{\rm inc}\times\nu,
\end{equation}
where $\nu$ is the unit normal vector on $S$. 

The potential functions $\phi$ and $\boldsymbol\psi$ are required to be
quasi-periodic in $x_1$ and $x_2$ directions with periods $\Lambda_1$ and 
$\Lambda_2$. Hence they have the Fourier series expansions
\begin{equation}\label{FourierHelmholtz}
\phi(\boldsymbol{x})=\sum\limits_{n\in\mathbb{Z}^2}
\phi_n(x_3) e^{{\rm
i}\boldsymbol{\alpha}_n\cdot\boldsymbol{r}}, \quad
\boldsymbol{\psi}(\boldsymbol{x})=\sum\limits_{n\in\mathbb{Z}^2}
\boldsymbol{\psi}_n(x_3)
e^{{\rm i}\boldsymbol{\alpha}_n\cdot\boldsymbol{r}}.
\end{equation}
Plugging \eqref{FourierHelmholtz} into \eqref{Helmholtz1} and using the bounded
outgoing wave condition, we have from a straightforward calculation that $\phi$
and $\boldsymbol{\psi}$ admit the following expansions:
\begin{eqnarray}\label{Helmholtz_decay1}
\phi(\boldsymbol{x})=\sum\limits_{n\in\mathbb{Z}^2}
\phi_n(h)e^{{\rm
i}\left(\boldsymbol{\alpha}_n\cdot\boldsymbol{r}
+\beta_{1 n}(x_3-h)\right)},\quad
\boldsymbol{\psi}(\boldsymbol{x})=\sum\limits_{n\in\mathbb{Z}^2}
\boldsymbol{\psi}_n(h) e^{{\rm
i}\left(\boldsymbol{\alpha}_n\cdot\boldsymbol{r} +\beta_{2 n}(x_3-h)\right)},
\quad x_3>h,
\end{eqnarray}
where
\begin{equation}\label{beta}
\beta_{j n}=
\begin{cases}
\big(\kappa_j^2-|\boldsymbol{\alpha}_n|^2\big)^{1/2}
\quad &\text{if} ~ |\boldsymbol{\alpha}_n|<\kappa_j,\\
{\rm i}\big(|\boldsymbol{\alpha}_n|^2-\kappa_j^2\big)^{1/2}
\quad &\text{if} ~ |\boldsymbol{\alpha}_n|>\kappa_j. 
\end{cases}
\end{equation}
It follows from \eqref{FourierHelmholtz}--\eqref{Helmholtz_decay1} that
\begin{equation}\label{Helmholtz_decay}
\phi_n(x_3)=\phi_n(h)e^{{\rm i}\beta_{1 n}(x_3-h)},\quad
\psi_{j n}(x_3)=\psi_{j n}(h)e^{{\rm i}\beta_{2 n}(x_3-h)}, \quad j=1,2,3.
\end{equation}
We observe from \eqref{Helmholtz_decay1}--\eqref{beta} that $\beta_{j n}$ is a
pure imaginary number and thus $\phi_n$ and
$\boldsymbol{\psi}_n$ are known as surface wave modes 
when $|\boldsymbol{\alpha}_n|>\kappa_j$. 

Substituting \eqref{Helmholtz_decay1} into \eqref{Helmholtz}, we obtain the
representation of the scattered field $\boldsymbol u$ in terms of the Fourier
coefficients of the potential functions $\phi$ and $\boldsymbol\psi$:
\begin{eqnarray}\label{u1}
\boldsymbol{u}(\boldsymbol{x})&=&{\rm
i}\sum\limits_{n\in\mathbb{Z}^2}\left\{\begin{bmatrix}
\alpha_{1 n}\\ 
\alpha_{2 n}\\ \beta_{1 n}\end{bmatrix}
\phi_n(h)e^{{\rm i}
\beta_{1 n}(x_3-h)}\right.\notag\\
&&\qquad \left.+\begin{bmatrix}
\alpha_{2 n}\psi_{3 n}(h)-\beta_{2 n}
\psi_{2 n}(h) \\
\beta_{2 n}\psi_{1 n}(h)-\alpha_{1 n}
\psi_{3 n}(h)\\
\alpha_{1 n}\psi_{2 n}(h)-\alpha_{2 n}\psi_{1 n}(h)
\end{bmatrix}e^{{\rm i}\beta_{2 n}(x_3-h)}\right\}e^{{\rm
i}\boldsymbol{\alpha}_n\cdot\boldsymbol{r}}. 
\end{eqnarray}
Noting $\nabla\cdot \boldsymbol{\psi}=0$, we may represent conversely the
coefficients of the potential functions of $\phi$ and $\boldsymbol{\psi}$ by
the coefficients of the scattered field $\boldsymbol{u}=(u_1, u_2, u_3)^\top$: 
\begin{eqnarray}
\phi_n(h)&=&-\frac{\rm i}{\chi_n}\left(\alpha_{1 n}u_{1 n}(h)
+\alpha_{2 n}u_{2 n}(h)+\beta_{2 n}u_{3 n}(h)\right), \label{phi}\\
\psi_{1 n}(h)&=&-\frac{\rm i}{\chi_n}\bigg(\frac{1}{\kappa_2^2}\alpha_{1
n}\alpha_{2 n} \left(\beta_{1 n}-\beta_{2 n}\right)
u_{1 n}(h)-\alpha_{2 n}u_{3 n}(h)\notag\\
&&\quad+\frac{1}{\kappa_2^2}\left[\alpha_{1 n}^2\beta_{2 n}
+\alpha_{2 n}^2 \beta_{1 n}+\beta_{1 n}\beta_{2
n}^2\right]u_{2 n}(h)\bigg),\label{psi1}\\
\psi_{2 n}(h)&=&-\frac{\rm
i}{\chi_n}\bigg(-\frac{1}{\kappa_2^2}\left[\alpha_{1 n}^2
\beta_{1 n}+\alpha_{2 n}^2\beta_{2 n}
+\beta_{1 n}\beta_{2 n}^2\right]u_{1 n}(h) \notag\\
&&\quad -\frac{1}{\kappa_2^2}\alpha_{1 n}\alpha_{2 n}\left(\beta_{1
n}-\beta_{2 n}\right)u_{2 n}(h)+\alpha_{1 n}u_{3 n} (h)\bigg),\label {psi2}\\
\psi_{3 n}(h)&=&-\frac{\rm
i}{\kappa_2^2}\left(\alpha_{2 n}u_{1 n}(h)-\alpha_{1 n}
u_{2 n}(h)\right),\label{psi3}	
\end{eqnarray}
where $\chi_n=\left|\boldsymbol{\alpha}_n\right|^2+\beta_{1 n}
\beta_{2 n}$. It is easy to verify that $\chi_n\neq 0$ for
$n\in\mathbb Z^2$. 

Define a differential operator 
\begin{equation}\label{OperatorD}
D\boldsymbol{u}=\mu\partial_{x_3}\boldsymbol{u}
+(\lambda+\mu)\left(\nabla\cdot\boldsymbol{u}\right)\boldsymbol e_3\quad
\text{on} ~ \Gamma_h,
\end{equation}
where $\boldsymbol e_3=(0, 0, 1)^\top$. Substituting \eqref{u1}--\eqref{psi3}
into the differential operator $D$, we may deduce the DtN operator
\begin{equation}\label{DtN}
T\boldsymbol{u}=\sum\limits_{n\in\mathbb{Z}^2} M_n\boldsymbol{u}_n(h)e^{{\rm
i}\boldsymbol{\alpha}_n\cdot\boldsymbol{r}},
\end{equation}
where the matrix $M_n$ is defined as
\begin{align}
M_n=\frac{{\rm i}\mu}{\chi_n}
\begin{bmatrix}
\alpha_{1 n}^2\beta^{(n)}_{12}+\beta_{2 n}\chi_n &
\alpha_{1 n}\alpha_{2 n}\beta^{(n)}_{12} &
\alpha_{1 n}\beta_{2 n}\beta^{(n)}_{12}\\
\alpha_{1 n}\alpha_{2 n}\beta^{(n)}_{12}&
\alpha_{2 n}^2\beta^{(n)}_{12}+\beta_{2 n}
\chi_n &
\alpha_{2 n}\beta_{2 n}\beta^{(n)}_{12} \\
-\alpha_{1 n}\beta_{2 n}\beta^{(n)}_{12}
&-\alpha_{2 n}\beta_{2 n}\beta^{(n)}_{12}
&\kappa_2^2 \beta_{2 n}
\end{bmatrix}.\label{Mn}
\end{align}
Here $\beta^{(n)}_{12}=\beta_{1 n}-\beta_{2 n}$. The details can be found in
\cite{JLLZ-cms18} for the derivation.

Based on the DtN operator \eqref{DtN}, the scattering problem
\eqref{ewe}--\eqref{bc} can be equivalently reduced to the following boundary
value problem:
\begin{equation}\label{BVPu}
\begin{cases}
\mu\Delta\boldsymbol{u}+(\lambda+\mu)\nabla\nabla\cdot\boldsymbol{u}
+\omega^2\boldsymbol{u}=0 \qquad &\text{in} ~ \Omega, \\
D\boldsymbol{u}=T\boldsymbol{u}  & \text{on} ~ \Gamma_h,\\
\boldsymbol{u}=-\boldsymbol{u}^{\rm inc} &\text{on} ~ S.
\end{cases}
\end{equation}
The variational problem of \eqref{BVPu} is to find $\boldsymbol{u}\in
\boldsymbol{H}_{\rm qp}^{1}(\Omega)$ with 
$\boldsymbol{u}=-\boldsymbol{u}^{\rm inc}$ on $S$ such that
\begin{equation}\label{variation1}
a(\boldsymbol{u}, \boldsymbol{v})=0 \quad \forall\, \boldsymbol{v}\in
\boldsymbol{H}_{S, {\rm qp}}^{1}(\Omega),
\end{equation}
where the sesquilinear form $a: H_{\rm qp}^{1}(\Omega)\times
H_{\rm qp}^{1}(\Omega)\to \mathbb{C} $ is
\[
a(\boldsymbol{u}, \boldsymbol{v})=\mu\int_{\Omega}
\nabla\boldsymbol{u}:\nabla\overline{\boldsymbol{v}}{\rm d}\boldsymbol{x}
+(\lambda+\mu)\int_{\Omega}\left(\nabla\cdot\boldsymbol{u}
\right)\left(\nabla\cdot\overline{\boldsymbol{v}}\right){\rm d}\boldsymbol{x}
-\omega^2\int_{\Omega}\boldsymbol{u} \cdot\overline{\boldsymbol{v}}{\rm
d}\boldsymbol{x}-\int_{\Gamma_h}T
\boldsymbol{u}\cdot\overline{\boldsymbol{v}}{\rm d}s.
\]
Here $A:B={\rm tr}(AB^\top)$ is the Frobenius inner product of two square
matrices $A$ and $B$.

The well-posedness of the variational problem \eqref{variation1}  was discussed
in \cite{EH-2012}. It was shown that the variational problem has a unique weak
solution for all but a discrete set of frequencies. Here we simply assume that
the variational problem \eqref{variation1} admits a unique solution which
satisfies the estimates
\[
\|\boldsymbol{u}\|_{\boldsymbol{H}^1(\Omega)}\lesssim \|\boldsymbol{u}^{\rm
inc}\|_{\boldsymbol{H}^{1/2}(S)}
\lesssim \|\boldsymbol{u}^{\rm inc}\|_{\boldsymbol{H}^{1}(\Omega)}.
\]
By the general theory of Babu\v{s}ka and Aziz \cite{BA-1973}, there exists a
$\gamma>0$ such that the following inf-sup condition holds: 
\begin{equation*}
\sup\limits_{0\neq \boldsymbol{v}\in
\boldsymbol{H}_{\rm qp}^{1}(\Omega)}\frac{\left|a(\boldsymbol{u},
\boldsymbol{v})\right|} {\|\boldsymbol{v}\|_{\boldsymbol{H}^1(\Omega)}}\geq
\gamma \|\boldsymbol{u}\|_{\boldsymbol{H}^1(\Omega)}\quad
\forall\, \boldsymbol{u}\in \boldsymbol{H}_{\rm qp}^{1}(\Omega).
\end{equation*}

\section{The finite element approximation}\label{section: DP}

Let $\mathcal M_h$ be a regular tetrahedral mesh of the domain $\Omega$, where
$h$ denotes the maximum diameter of all the elements in $\mathcal M_h$.
To handle the quasi-periodic solution, we assume that the mesh is periodic in
both $x_1$ and $x_2$ directions, i.e., the surface meshes on the planes $x_1=0$
and $x_2=0$ coincide with the surface meshes on the planes $x_1=\Lambda_1$
and $x_2=\Lambda_2$, respectively. We also assume for simplicity that $S$ is
polygonal to keep from using the isoparametric finite element space and deriving
the approximation error of the boundary $S$ in order to avoid being distracted
from the main focus of the a posteriori error analysis. 

Let $V_h\subset \boldsymbol{H}_{\rm qp}^{1}(\Omega)$ be a conforming finite
element space, i.e.
\[
V_h=\{\boldsymbol{v}\in C_{\rm qp}(\Omega)^3:
\boldsymbol{v}|_{K}\in P_m(K)^3 ~ \forall K\in \mathcal M_h\},
\]
where $C_{\rm qp}(\Omega)$ is the set of all continuous functions satisfying
the quasi-periodic boundary condition, $m$ is a positive integer, and $P_m$
denotes the set of all polynomials with degree no more than $m$. 

Since the non-local DtN operator \eqref{DtN} is given as an infinite series, in
practice, it needs to be truncated into a sum of finitely many terms
\begin{equation}\label{dDtN}
T_{N}\boldsymbol{u}=\sum\limits_{|n_1|,|n_2|\leq N}
M_n\boldsymbol{u}_n(h) e^{{\rm
i}\boldsymbol{\alpha}_n\cdot\boldsymbol{r}}\quad \text{on} ~ 
\Gamma_h,
\end{equation}
where $N>0$ is a sufficiently large integer. Using \eqref{dDtN}, we arrive at
the truncated finite element approximation: find $\boldsymbol{u}_N^h\in V_h$
such that $\boldsymbol{u}_N^{h}=-\boldsymbol{u}^{\rm inc}$ on $S$ and satisfies
the variational problem
\begin{equation}\label{variation2}
a_N(\boldsymbol{u}_N^h, \boldsymbol{v}^h)=0\quad \forall\,
\boldsymbol{v}^h\in V_{h,S},
\end{equation}
where $V_{h,S}=\left\{\boldsymbol{v}\in V_h: \boldsymbol{v}=0 \,{\rm
on}\,S\right\}$ and the sesquilinear form 
$a_N: V_h\times V_h\rightarrow \mathbb{C}$ is 
\begin{align*}
a_N(\boldsymbol{u}, \boldsymbol{v})=\mu\int_{\Omega}
\nabla\boldsymbol{u}:\nabla\overline{\boldsymbol{v}}{\rm d}\boldsymbol{x}
+(\lambda+\mu)\int_{\Omega}\left(\nabla\cdot\boldsymbol{u}
\right)\left(\nabla\cdot\overline{\boldsymbol{v}}\right){\rm d}\boldsymbol{x}
-\omega^2\int_{\Omega}\boldsymbol{u}\cdot \overline{\boldsymbol{v}}{\rm
d}\boldsymbol{x}-\int_{\Gamma_h}T_N\boldsymbol{u}\cdot
\overline{\boldsymbol{v}}{\rm d}s.
\end{align*}

By \cite{S-mc74}, the discrete inf-sup condition of the sesquilinear form $a_N$
can be established for sufficient large $N$ and small enough $h$. Based on the
general theory in \cite{BA-1973}, it can be shown that the discrete variational
problem \eqref{variation2}  has a unique solution $\boldsymbol{u}_N^h\in V_h$.
The details are omitted for brevity since our focus is on the a posteriori
error estimate. 

\section{The a posteriori error analysis}\label{section: PEA}

For any  tetrahedral element $K\in\mathcal M_h$, denote by $h_K$ its diameter.
Define the operator residual in $K$ as
\[
R_{K}\boldsymbol{u}=\big(\mu\Delta\boldsymbol{u}
+(\lambda+\mu)\nabla\nabla\cdot\boldsymbol{u}
+\omega^2 \boldsymbol{u}\big)|_{K}.
\]
Let $\mathcal{F}_h$ be the set of all the faces on $\mathcal M_h$. Given any
interior face $F\in \mathcal{F}_h$, which is the common face of tetrahedral 
element $K_1$ and $K_2$, we define the jump residual across $F$ as
\[
J_{F}\boldsymbol{u}=\mu\nabla\boldsymbol{u}|_{K_1}
\cdot\nu_1+(\lambda+\mu)\left(\nabla\cdot\boldsymbol{u}|_{K_1}\right)\nu_1
+\mu\nabla\boldsymbol{u}|_{K_2}
\cdot\nu_2+(\lambda+\mu)\left(\nabla\cdot\boldsymbol{u}|_{K_2}\right)\nu_2,
\]
where $\nu_j$, $j=1,2$ is the unit outward normal vector on the face of $K_j$.
For any boundary face $F\in \mathcal{F}_h\cap \Gamma_h$, define the jump
residual as
\[
J_{F}\boldsymbol{u}=2\left(T_N\boldsymbol{u}-D\boldsymbol{u}\right).
\]

Denote the four lateral boundary surfaces by 
\begin{eqnarray*} 
&& \Gamma_{10}=\left\{\boldsymbol{x}\in\mathbb{R}^3: x_1=0, 0<x_2<\Lambda_2,
f(0,x_2)<x_3<h\right\},\\
&& \Gamma_{11}=\left\{\boldsymbol{x}\in\mathbb{R}^3: x_1=\Lambda_1,
0<x_2<\Lambda_2, f(\Lambda_1,x_2)<x_3<h\right\},\\
&& \Gamma_{20}=\left\{\boldsymbol{x}\in\mathbb{R}^3: 0<x_1<\Lambda_1, x_2=0,
f(x_1, 0)<x_3<h\right\},\\
&& \Gamma_{21}=\left\{\boldsymbol{x}\in\mathbb{R}^3: 0<x_1<\Lambda_1,
x_2=\Lambda_2, f(x_1, \Lambda_2)<x_3<h\right\}.
\end{eqnarray*}
For any boundary face $F\in \Gamma_{10}$ and the corresponding face
$F'\in\Gamma_{11}$, if $F\in K_1$ and $F'\in K_2$, then the jump residual
is defined as
\begin{eqnarray*}
J_{F}^{(1)}\boldsymbol{u}&=&\left[\mu\partial_{x_1}\boldsymbol{u}|_{K_1}
+(\lambda+\mu)(\nabla\cdot\boldsymbol{u}|_{K_1}) \boldsymbol e_1 
\right]-e^{-{\rm i}\alpha_1\Lambda_1}\left[\mu\partial_{x_1}\boldsymbol{u}
|_{K_2}+(\lambda+\mu) (\nabla\cdot\boldsymbol{u}|_{K_2})\boldsymbol e_1
\right],\\
J_{F'}^{(1)}\boldsymbol{u}&=&e^{{\rm
i}\alpha_1\Lambda_1}\left[\mu\partial_{x_1} \boldsymbol{u}|_{K_1}
+(\lambda+\mu)(\nabla\cdot\boldsymbol{u}|_{K_1})\boldsymbol e_1
\right]-\left[\mu\partial_{x_1}\boldsymbol{u}|_{K_2}
+(\lambda+\mu)(\nabla\cdot\boldsymbol{u}|_{K_2})\boldsymbol e_1 \right],
\end{eqnarray*}
where $\boldsymbol e_1=(1, 0, 0)^\top$. Similarly, for any face $F\in
\Gamma_{20}$ and its corresponding face $F'\in \Gamma_{21}$, the jump residual
is defined as
\begin{eqnarray*}
J_{F}^{(2)}\boldsymbol{u}&=&\left[\mu\partial_{x_2}\boldsymbol{u}|_{K_1}
+(\lambda+\mu) (\nabla\cdot\boldsymbol{u}|_{K_1}) \boldsymbol e_2 
\right]-e^{-{\rm
i}\alpha_2\Lambda_2}\left[\mu\partial_{x_2}\boldsymbol{u}|_{K_2} +(\lambda+\mu)
(\nabla\cdot\boldsymbol{u}|_{K_2})\boldsymbol e_2\right],\\
J_{F'}^{(2)}\boldsymbol{u}&=&e^{{\rm
i}\alpha_2\Lambda_2}\left[\mu\partial_{x_2}\boldsymbol{u}|_{K_1}
+(\lambda+\mu)(\nabla\cdot\boldsymbol{u}|_{K_1}) \boldsymbol e_2
\right]-\left[\mu\partial_{x_2}
\boldsymbol{u}|_{K_2}+(\lambda+\mu)(\nabla\cdot\boldsymbol{u}|_{K_2})\boldsymbol
e_2 \right],
\end{eqnarray*}
where $\boldsymbol e_2=(0, 1, 0)^\top$.

For any tetrahedral element $K\in\mathcal M_h$, denote by $\eta_{K}$ the local
error estimator as follows:
\[
\eta_K^2 =h_K^2 \|R_{K}
\boldsymbol{u}\|_{\boldsymbol{L}^2(K)}^2+h_K\sum\limits_{F\subset\partial
K}\left(\|J_F^{(1)}\boldsymbol{u}\|^2_{\boldsymbol{L}^2(F)}+\|J_F^{(2)}
\boldsymbol{u}\|^2_{\boldsymbol{L}^2(F)}\right).
\]

For convenience, we introduce a weighted $\boldsymbol{H}^1(\Omega)$ norm 
\begin{equation}\label{weightnorm}
\vvvert 
\boldsymbol{u}\vvvert_{\boldsymbol{H}^1(\Omega)}^2=\mu\int_{\Omega}|\nabla
\boldsymbol{u}|^2 {\rm d}\boldsymbol{x}
+(\lambda+\mu)\int_{\Omega} |\nabla\cdot\boldsymbol{u}|^2 {\rm d}\boldsymbol{x}
+\omega^2\int_{\Omega} |\boldsymbol{u}|^2 {\rm d}\boldsymbol{x}.
\end{equation}
Since $\mu$ and $\lambda+\mu$ are positive, it is easy to check that 
\begin{equation*}
\min  \left(\mu, \omega^2\right)
\|\boldsymbol{u}\|_{\boldsymbol{H}^1(\Omega)}^2\leq
\vvvert \boldsymbol{u}\vvvert_{\boldsymbol{H}^{1}(\Omega)}^2\leq
\max \left(2\lambda+3\mu, \omega^2\right)
\|\boldsymbol{u}\|_{\boldsymbol{H}^1(\Omega)}^2
\quad \forall\, \boldsymbol{u}\in\boldsymbol{H}^{1}(\Omega), 
\end{equation*}
which implies that the weighted $\boldsymbol{H}^1(\Omega)$ norm
\eqref{weightnorm} is equivalent to the standard $\boldsymbol{H}^1(\Omega)$
norm.

The following theorem is the main result of the paper. It presents the a
posteriori error estimate between the solutions of the original scattering
problem \eqref{variation1} and the truncated finite element approximation
\eqref{variation2}. 

\begin{theorem}\label{mainthm}
Let $\boldsymbol{u}$ and $\boldsymbol{u}_N^h$ be the solutions of the
variational problems \eqref{variation1} and \eqref{variation2}, 
respectively. Then for any $\hat h$ such that
$\max_{\boldsymbol{r}\in\mathbb{R}^2} f(\boldsymbol{r})<\hat h<h$ and
for sufficiently large $N$, the following a posteriori error estimate holds:
\begin{equation}\label{mainresult}
\vvvert \boldsymbol{u}-\boldsymbol{u}_N^h\vvvert
_{\boldsymbol{H}^{1}(\Omega)}\lesssim
\left(\sum_{K\in\mathcal M_h}\eta_K^2\right)^{1/2}+\max_{|n|_{\min}>N}
\left(|n|_{\rm
max}e^{-|\beta_{2 n}|(h-\hat h)}\right)\|\boldsymbol {u}^{\rm
inc}\|_{\boldsymbol{H}^{1}(\Omega)},
\end{equation}
where $|n|_{\min}=\min (|n_1|, |n_2|)$ and $|n|_{\max}=\max(|n_1|,
|n_2|)$.
\end{theorem}

The a posteriori error \eqref{mainresult} contains two parts: the finite element
discretization error and the truncation error of the DtN operator. Since
$\hat h<h$, the latter is almost exponentially decaying. Hence the DtN truncated
error can be controlled to be small enough so that it does not contaminate the
finite element discretization error. 

To prove Theorem \ref{mainthm}, let us begin with the following trace result in
$\boldsymbol{H}_{\rm qp}^{1}(\Omega)$. The proof can be found
in \cite[Lemma 3.3]{JLLZ-cms18}. 

\begin{lemma}\label{lemma1}
Let $a=\min_{\boldsymbol r\in\mathbb R^2}f(\boldsymbol r)$. Then for any
$\boldsymbol{u}\in \boldsymbol{H}_{\rm qp}^{1}(\Omega)$ the following estimate
holds: 
\[
\|\boldsymbol{u}\|_{\boldsymbol{H}^{1/2}(\Gamma_{h})}\leq
C\|\boldsymbol{u}\|_{\boldsymbol{H}^{1}(\Omega)}, 
\]
where $C=(1+(h-a)^{-1})^{1/2}$. 
\end{lemma}

Denote by $\boldsymbol{\xi}=\boldsymbol{u}-\boldsymbol{u}_N^{h}$ the error
between the solutions of \eqref{variation1} and \eqref{variation2}, then a
simple calculation yields
\begin{eqnarray}\label{Error}
\vvvert \boldsymbol{\xi}\vvvert_{\boldsymbol{H}^{1}(\Omega)}^2 
&=&\mu\int_{\Omega}\nabla\boldsymbol{\xi}:\nabla\overline{\boldsymbol{\xi}}\,{
\rm d}\boldsymbol{x} +(\lambda+\mu)\int_{\Omega}(\nabla\cdot\boldsymbol{\xi}
)(\nabla\cdot\overline{\boldsymbol{\xi}})\,{\rm d}\boldsymbol{x}
+\omega^2\int_{\Omega}\boldsymbol{\xi}\cdot\overline{\boldsymbol{\xi}}\,{\rm
d}\boldsymbol{x}\notag\\
&=&  \Re a(\boldsymbol{\xi},
\boldsymbol{\xi})+2\omega^2\int_{\Omega}\boldsymbol{\xi}\cdot\overline{
\boldsymbol{\xi}}\,{\rm d}\boldsymbol{x}
+\Re\int_{\Gamma_{h}}T\boldsymbol{\xi}\cdot\overline
{\boldsymbol{\xi }}\,{\rm d}s\notag\\
&=&\Re a(\boldsymbol{\xi}, \boldsymbol{\xi})
+\Re\int_{\Gamma_{h}}
\left(T-T_N\right)\boldsymbol{\xi}\cdot\overline{\boldsymbol
{\xi}}\,{\rm  d}s +2\omega^2\int_{\Omega}
\boldsymbol{\xi}\cdot\overline{\boldsymbol{\xi}}\,{\rm
d}\boldsymbol{x}+\Re\int_{\Gamma_{h}}T_N\boldsymbol{\xi}
\cdot\overline{\boldsymbol{ \xi}}\,{\rm d}s.
\end{eqnarray}
Due to the equivalence of the weighted norm
$\vvvert\cdot\vvvert_{\boldsymbol{H}^{1}(\Omega)}$ to the standard norm
$\|\cdot\|_{\boldsymbol{H}^1(\Omega)}$, it suffices to estimate the four terms
on the right hand side of \eqref{Error} one by one. The estimates of the first
two terms are given in Lemmas \ref{lemma2} and \ref{lemma3}. 

\begin{lemma}\label{lemma2}
Let $\boldsymbol{u}\in \boldsymbol{H}_{\rm qp}^1(\Omega)$ be the solution of
variational problem \eqref{variation1}. For any $\boldsymbol{v}\in
\boldsymbol{H}_{\rm qp}^{1}(\Omega)$ and a sufficiently large $N$, the following
estimate holds: 
\[
\left|\int_{\Gamma_{h}}\left(T-T_N\right)\boldsymbol{u}\cdot\overline{
\boldsymbol{v}}\,{\rm d}s\right|\lesssim
\max\limits_{|n|_{\min}>N}\left(|n|_{\max}e^{-|\beta_{2 n}|(h-\hat h)}\right)
\|\boldsymbol{u}^{\rm inc}\|_{\boldsymbol{H}^1(\Omega)}
\|\boldsymbol{v}\|_{\boldsymbol{H}^{1}(\Omega)}.
\]
\end{lemma}

\begin{proof}
It follows from \eqref{Helmholtz_decay} that we have
\begin{equation}\label{Helmholtz_decay2}
\phi_n(h)=\phi_n(\hat h)e^{{\rm
i}\beta_{1 n}(h-\hat h)},\quad
\psi_{j n}(h)=\psi_{j n}(\hat h)e^{{\rm
i}\beta_{2 n}(h-\hat h)}, \quad j=1,2,3.
\end{equation}
Substituting \eqref{Helmholtz_decay2} into \eqref{u1}, we obtain the
Fourier coefficients of $\boldsymbol{u}$ at $x_3=h$ in terms of the Fourier
coefficients of $\phi$ and $\boldsymbol\psi$ at $x_3=\hat h$:
\begin{eqnarray}
&&\begin{bmatrix}u_{1 n}(h)\\
u_{2 n}(h)\\  u_{3 n}(h) \\0 \end{bmatrix}=
{\rm i}\begin{bmatrix}
\alpha_{1 n} & 0 & -\beta_{2 n} &
\alpha_{2 n}\\
\alpha_{2 n} & \beta_{2 n} & 0 &
-\alpha_{1 n}\\
\beta_{1 n} & -\alpha_{2 n} &
\alpha_{1 n} & 0\\
0 & \alpha_{1 n} & \alpha_{2 n} &
\beta_{2 n}
\end{bmatrix}\begin{bmatrix}\phi_n(h) \\
\psi_{1 n}(h)\\
\psi_{2 n}(h)  \\
\psi_{3 n}(h)\end{bmatrix} \notag\\
&&={\rm i}\begin{bmatrix}
\alpha_{1 n} & 0 & -\beta_{2 n} &
\alpha_{2 n}\\
\alpha_{2 n} & \beta_{2 n} & 0 &
-\alpha_{1 n}\\
\beta_{1 n} & -\alpha_{2 n} &
\alpha_{1 n} & 0\\
0 & \alpha_{1 n} & \alpha_{2 n} &
\beta_{2 n}
\end{bmatrix}
{\rm diag} \left(\begin{bmatrix}
e^{{\rm i}\beta_{1 n}(h-\hat h)}\\ e^{{\rm
i}\beta_{2 n}(h-\hat h)} \\
e^{{\rm i}\beta_{2 n}(h-\hat h)}\\e^{{\rm
i}\beta_{2 n}(h-\hat h)}
\end{bmatrix}\right)
\begin{bmatrix}\phi_n(\hat h) \\
\psi_{1 n}(\hat h)\\
\psi_{2 n}(\hat h) \\
\psi_{3 n}(\hat h)\end{bmatrix}  \notag \\
&& :={\rm i} A_n (\phi_n(\hat h),
\psi_{1 n}(\hat h), \psi_{2 n}(\hat h),
\psi_{3 n}(\hat h))^\top. \label{uphi}
\end{eqnarray}
Replacing $h$ by $\hat h$ in \eqref{phi}--\eqref{psi3}, we may equivalently
have the matrix form
\begin{equation}\label{phiu}
\begin{bmatrix}
\phi_n(\hat h) \\
\psi_{1 n}(\hat h)\\
\psi_{2 n}(\hat h)  \\
\psi_{3 n}(\hat h)\end{bmatrix}
=-\frac{\rm i}\chi_n B_n
\begin{bmatrix}u_{1 n}(\hat h)\\
u_{2 n}(\hat h)\\  u_{3 n}(\hat h)\end{bmatrix},
\end{equation}
where the entries of the $4\times 3$ matrix $B_n$ are
\begin{eqnarray*}
&&B^{(n)}_{11}=	\alpha_{1 n},\quad
B^{(n)}_{12}=\alpha_{2 n}, \quad B^{(n)}_{13}=\beta_{2 n},\quad B^{(n)}_{23}=
-\alpha_{2 n},\quad B^{(n)}_{33}=\alpha_{1 n},\quad B^{(n)}_{43}=0,\\
&&B^{(\boldsymbol{n})}_{21}=-B^{(n)}_{32}=\frac{1}{\kappa_2^2}\alpha_{1 n}
\alpha_{2 n}\left(\beta_{1 n} -\beta_{2 n}\right), \quad 
B^{(n)}_{41}=\frac{1}{\kappa_2^2}\alpha_{2 n}\chi_n, \quad
B^{(n)}_{42}=-\frac{1}{\kappa_2^2}\alpha_{1 n} \chi_n,\\
&& B^{(n)}_{22}=\frac{1}{\kappa_2^2}\left(\alpha_{1 n}^2\beta_{2 n}+\alpha_{2
n}^2 \beta_{1 n}+\beta_{1 n}\beta_{2 n}^2\right),\quad 
B^{(n)}_{31}=-\frac{1}{\kappa_2^2}\left(\alpha_{1 n}^2\beta_{1 n}+\alpha_{2 n}^2
\beta_{2 n}+\beta_{1 n} \beta_{2 n}^2\right). 
\end{eqnarray*}
Plugging \eqref{uphi} into \eqref{phiu} yields 
\begin{eqnarray}\label{decayU}
\begin{bmatrix}u_{1 n}(h)\\
u_{2 n}(h)\\  u_{3 n}(h)\end{bmatrix}
=\frac{1}{\chi_n}(A_n B_n)|_{3\times 3}  
\begin{bmatrix}u_{1 n}(\hat h)\\
u_{2 n}(\hat h)\\ 
u_{3 n}(\hat h)\end{bmatrix}
:=\frac{1}{\chi_n}P_n
\begin{bmatrix}u_{1 n}(\hat h)\\ 
u_{2 n}(\hat h)\\  u_{3 n}(\hat h)\end{bmatrix},
\end{eqnarray}
where $(A_n B_n)|_{3\times 3}$ is the leading principal submatrix of order 3 of
the matrix $A_n B_n$. A straight forward computation yields that
\begin{eqnarray*}
P_{11}^{(n)}&=&\alpha_{1n}^2 e^{{\rm i}\beta_{1 n}(h-\hat h)}
+\frac{1}{\kappa_2^2}\alpha_{2 n}^2 \chi_n e^{{\rm i}\beta_{2 n}(h-\hat h)}\\
&&+\frac{1}{\kappa_2^2}\left( \alpha_{1 n}^2\beta_{1 n}
+\alpha_{2 n}^2\beta_{2 n}+\beta_{1 n}\beta_{2 n}^2\right)
\beta_{2 n} e^{{\rm i}\beta_{2 n}(h-\hat h)},\\
P_{12}^{(n)}&=&\alpha_{1 n}\alpha_{2 n}e^{{\rm
i}\beta_{1 n}(h-\hat h)} +\frac{1}{\kappa_2^2}\alpha_{1 n}
\alpha_{2 n}\left(\beta_{1 n}-\beta_{2 n}\right)\beta_{2 n}
e^{{\rm i}\beta_{2 n}(h-\hat h)}\\
&&-\frac{1}{\kappa_2^2}\alpha_{1 n} \alpha_{2 n}\chi_n
e^{{\rm i}\beta_{2 n}(h-\hat h)},\\
P_{13}^{(n)}&=&\alpha_{1 n}\beta_{2 n}\big(e^{{\rm
i}\beta_{1 n}(h-\hat h)}-e^{{\rm
i}\beta_{2 n}(h-\hat h)}\big),
\end{eqnarray*}

\begin{eqnarray*}
P_{21}^{(n)}&=&\alpha_{1 n}\alpha_{2 n}e^{{\rm
i}\beta_{1 n}(h-\hat h)} +\frac{1}{\kappa_2^2}\alpha_{1 n}
\alpha_{2 n}\left(\beta_{1 n}-\beta_{2 n}
\right)\beta_{2 n} e^{{\rm i}\beta_{2 n}(h-\hat h)}\\
&&-\frac{1}{\kappa_2^2}\alpha_{1 n} \alpha_{2 n}\chi_n
e^{{\rm i}\beta_{2 n}(h-\hat h)},\\
P_{22}^{(n)}&=&\alpha_{2 n}^2 e^{{\rm i}\beta_{1 n}(h-\hat h)}
+\frac{1}{\kappa_2^2}\alpha_{1 n}^2 \chi_n e^{{\rm i}\beta_{2 n}(h-\hat h)}\\
&&+\frac{1}{\kappa_2^2}\left(\alpha_{1 n}^2\beta_{2 n}
+\alpha_{2 n}^2\beta_{1 n}+\beta_{1 n}\beta_{2 n}^2\right)
\beta_{2 n} e^{{\rm i}\beta_{2 n}(h-\hat h)},\\
P_{23}^{(n)}&=&\alpha_{2 n}\beta_{2 n}\big(e^{{\rm
i}\beta_{1 n}(h-\hat h)}-e^{{\rm i}\beta_{2 n}(h-\hat h)}\big),
\end{eqnarray*}

\begin{eqnarray*}
P_{31}^{(n)}&=&\alpha_{1 n}\beta_{1 n}e^{{\rm
i}\beta_{1 n}(h-\hat h)} -\frac{1}{\kappa_2^2}\alpha_{1 n}\alpha_{2
n}^2\left(\beta_{1 n}-\beta_{2 n}\right)e^{{\rm i}\beta_{2 n}(h-\hat h)}\\
&&-\frac{1}{\kappa_2^2}\left(\alpha_{1 n}^2\beta_{1 n}
+\alpha_{2 n}^2\beta_{2 n}+\beta_{1 n}\beta_{2 n}^2\right)
\alpha_{1 n} e^{{\rm i}\beta_{2 n}(h-\hat h)},\\
P_{32}^{(n)}&=&\alpha_{2 n}\beta_{1 n}e^{{\rm
i}\beta_{1 n}(h-\hat h)}
-\frac{1}{\kappa_2^2}\alpha_{2 n}\alpha_{1 n}^2\left(\beta_{1 n}-\beta_{2
n}\right) e^{{\rm i}\beta_{2 n}(h-\hat h)}\\
&&-\frac{1}{\kappa_2^2}\left(\alpha_{1 n}^2\beta_{2 n}
+\alpha_{2 n}^2\beta_{1 n}+\beta_{1 n}\beta_{2 n}^2\right)\alpha_{2 n}
e^{{\rm i}\beta_{2 n}(h-\hat h)},\\
P_{33}^{(n)}&=&\beta_{1 n}\beta_{2 n}
e^{{\rm i}\beta_{1 n}(h-\hat h)}
+(\alpha_{1 n}^2 + \alpha_{2 n}^2) e^{{\rm i}\beta_{2 n}(h-\hat h)}.
\end{eqnarray*}

When
$|\boldsymbol{\alpha}_n|^2=\alpha_{1 n}^2+\alpha_{2
n}^2>\kappa_2^2$, it follows from \eqref{beta} that both $\beta_{1 n}$ and 
$\beta_{2 n}$ are pure imaginary numbers. We may easily show 
\begin{eqnarray}
\frac{\kappa_2^2}{2}<\chi_n =
\left|\boldsymbol{\alpha}_n\right|^2
-\left(|\boldsymbol{\alpha}_n|^2-\kappa_1^2\right)^{
1/2}\left(|\boldsymbol{\alpha}_n|^2-\kappa_2^2\right)^{1/2}<
\kappa_1^2+\kappa_2^2. \label{chi}
\end{eqnarray}
and
\begin{eqnarray}
{\rm i}(\beta_{2n}-\beta_{1n})=
\left(|\boldsymbol{\alpha}_n|^2-\kappa_1^2\right)^
{1/2}-\left(|\boldsymbol{\alpha}_n|^2-\kappa_2^2\right)^{1/2}
<\frac{\kappa_2^2-\kappa_1^2}{2(|\boldsymbol{\alpha}_n|^2-\kappa_2^2)^{1/2}}
.\label{MVT}
\end{eqnarray}
Plugging \eqref{beta} and \eqref{chi}--\eqref{MVT} into
$P_n$, we obtain 
\begin{eqnarray*}
P_{11}^{(n)}
&=&\alpha_{1 n}^2 e^{{\rm i}\beta_{1 n}(h-\hat h)}
+\frac{1}{\kappa_2^2}e^{{\rm i}\beta_{2 n}(h-\hat h)}
\Big\{|\boldsymbol{\alpha}_n|^2\beta_{1 n}\beta_{2 n}
-\alpha_{2 n}^2 (|\boldsymbol{\alpha}_n|^2-\kappa_2^2 )\\
&&-\beta_{1 n}\beta_{2 n} (|\boldsymbol{\alpha}_n|^2-\kappa_2^2 )
+\alpha_{2 n}^2|\boldsymbol{\alpha}_n|^2\Big\}\\
&=&\alpha_{1 n}^2 e^{{\rm i}\beta_{1 n}(h-\hat h)}
+e^{{\rm i}\beta_{2 n}(h-\hat h)} (\alpha_{2 n}^2
+\beta_{1 n}\beta_{2 n})\\
&=&\alpha_{1 n}^2\big(e^{{\rm i}\beta_{1 n}(h-\hat h)}
-e^{{\rm i}\beta_{2 n}(h-\hat h)}\big)
+e^{{\rm i}\beta_{2 n}(h-\hat h)}\chi_n,
\end{eqnarray*}
which gives 
\begin{eqnarray*}
|P_{11}^{(n)}|\lesssim |n|_{\max}e^{-|\beta_{2 n}|(h-\hat h)}.
\end{eqnarray*}
Similarly, we may show that all the entries of the matrix
$P_n$ have the estimates
\begin{equation}\label{aymptoticP}
|P_{ij}^{(n)}|\lesssim
|n|_{\max}e^{-|\beta_{2 n}|(h-\hat h)},\quad i,j=1,2,3.
\end{equation}
Substituting \eqref{chi} and \eqref{aymptoticP} into \eqref{decayU} gives
\begin{equation}\label{decayU2}
|\boldsymbol u_n(h)|^2\lesssim |n|^2_{\max} e^{-2|\beta_{2n}|(h-\hat h)}
|\boldsymbol u_n(\hat h)|^2.
\end{equation}

By \eqref{Mn}, it can be verified
from $\left|\boldsymbol{\alpha}_n\right|^2>\kappa_2^2$ that 
\begin{eqnarray}\label{aymM}
|M_{11}^{(n)}|&=&\left|\alpha_{1n}^2\left(\beta_{1 n}
-\beta_{2 n}\right)+\beta_{2 n}\chi_n\right| \notag\\
&=& \left|\alpha_{1 n}^2{\rm
i}\big( (|\boldsymbol{\alpha}_n|^2-\kappa_1^2)^{1/2}-
 (|\boldsymbol{\alpha}_n|^2-\kappa_2^2)^{1/2}\big)
+\beta_{2 n}\chi_n\right|\notag\\
&=& \left| |\alpha_{1n}|^2\frac{{\rm
i}(\kappa_2^2-\kappa_1^2)}{ (|\boldsymbol{\alpha}_n|^2
-\kappa_1^2)^{1/2}+(|\boldsymbol{\alpha}_n|^2-\kappa_2^2)^{1/2}}
+\beta_{2 n}\chi_n\right|\notag\\
&\lesssim& |n|_{\max}.
\end{eqnarray}
Following the same argument, we may show that 
\[
|M_{ij}^{(n)}|\lesssim |n|_{\max},\quad i,j=1,2,3.
\]

Substituting \eqref{decayU2}--\eqref{aymM} into \eqref{DtN}, we obtain
\begin{eqnarray*}
&&
\left|\int_{\Gamma_{h}}\left(T-T_N\right)\boldsymbol{u}
\cdot\overline{\boldsymbol{v}}\,{\rm d}s\right|
\leq\bigg|\Lambda_1\Lambda_2
\sum\limits_{|n|_{\min}>N}\left(M_n\boldsymbol{u}_n(h)\right)
\cdot\overline{\boldsymbol{v}(h)}\bigg|\\
&&\lesssim \Bigg(\sum\limits_{|n|_{\min}>N}
|n|_{\max}\left|\boldsymbol u_n(h)\right|^2\Bigg)^{1/2}
\Bigg(\sum\limits_{|n|_{\min}>N}
|n|_{\max}\left|\boldsymbol v_n(h)\right|^2\Bigg)^{1/2}\\
&&\lesssim \Bigg(\sum\limits_{|n|_{\min}>N}
|n|_{\max}^3 e^{-2|\beta_{2 n}|(h-\hat h)}
|\boldsymbol u_n(\hat
h)|^2\Bigg)^{1/2}\|\boldsymbol{v}\|_{\boldsymbol{H}^{1/2}(\Gamma_{h})}\\
&&\lesssim\max\limits_{|n|_{\min}>N}\left(|n|_{\max}e^{-|
\beta_{2 n}|(h-\hat h)}\right)
\|\boldsymbol{u}\|_{\boldsymbol{H}^{1}(\Omega)}\|\boldsymbol{v}\|_{
\boldsymbol{H}^{1}(\Omega)}\\
&&\lesssim \max\limits_{|n|_{\min}>N}\left(|n|_{\max}e^{-|\beta_{2 n}|
(h-\hat h)}\right)
\|\boldsymbol{u}^{\rm
inc}\|_{\boldsymbol{H}^{1}(\Omega)}\|\boldsymbol{v}\|_{\boldsymbol{H}^{1}
(\Omega)},
\end{eqnarray*}
which completes the proof. 
\end{proof}

\begin{lemma}\label{lemma3}
Let $\boldsymbol{v}$ be any function in $\boldsymbol{H}_{S, 
{\rm qp}}^{1}(\Omega)$, the following estimate holds:
\begin{eqnarray*}
\left|a(\boldsymbol{\xi}, \boldsymbol{v})+\int_{\Gamma_{h}}\left(
T-T_N\right)\boldsymbol{\xi}\cdot\overline{
\boldsymbol{v}}\,{\rm d}s\right|
&\lesssim& \Bigg(\bigg(\sum\limits_{K\in \mathcal M_h}\eta_K^2\bigg)^{1/2}\\
&&+\max\limits_{|n|_{\min}>N}\left(|n|_{\max}e^{-
|\beta_{2n}|(h-\hat h)}\right)\|\boldsymbol{u}^{\rm
inc}\|_{\boldsymbol{H}^{1}(\Omega)}\Bigg)\|\boldsymbol{v}\|_{\boldsymbol{H}^{1}
(\Omega)}.
\end{eqnarray*}
\end{lemma}

Since the proof of Lemma \ref{lemma3} is essentially the same as that for
\cite[Lemma 5.4]{LY-2019-periodic}, we omit it for brevity. The following two
lemmas are to estimate the last term in \eqref{Error}.

\begin{lemma}\label{lemma4}
Let $\hat{M}_n=-\frac{1}{2}\left(M_n+M_n^*\right)$. Then
$\hat{M}_n$ is positive definite for
$\left|\boldsymbol{\alpha}_n\right|>\kappa_2$.
\end{lemma}

\begin{proof}
A simple calculation shows that $\hat{M}_n=-M_n$ for 
$\left|\boldsymbol{\alpha}_n\right|>\kappa_2$.
It suffices to check Sylvester's criterion in order to prove that 
$\hat{M}_n$ is positive definite. First, it is easy to see that 
\[
\chi_n=|\boldsymbol{\alpha}_n|^2+\beta_{1 n}\beta_{2 n}
=|\boldsymbol{\alpha}_n|^2-\left(|\boldsymbol{
\alpha}_n|^2-\kappa_1^2\right)^{1/2}
\left(|\boldsymbol{\alpha}_n|^2-\kappa_2^2\right)^{1/2}>0.
\] 

Let $\hat{M}_{11}^{(n)}$ be the leading principal submatrix of order 1 for
$\hat{M}_n$. A simple calculation yields 
\begin{eqnarray}\label{M11}
\hat{M}_{11}^{(n)} &=& -\frac{{\rm
i}\mu}{\chi_n}\left[\alpha_{1 n}^2\left(\beta_{1 n}-\beta_{2 n}\right)+\beta_{2
n}\chi_n\right] \notag \\
&=&-\frac{{\rm i}\mu}{\chi_n}\left[
\alpha_{1 n}^2\left(\beta_{1 n}-\beta_{2 n}\right)+\beta_{2 n}
\left(|\boldsymbol{\alpha}_n|^2+\beta_{1 n}\beta_{2 n}\right)\right]\notag \\
&=&-\frac{{\rm i}\mu}{\chi_n}\left[
\alpha_{1 n}^2\beta_{1 n}+\alpha_{2 n}^2\beta_{2 n}
-\beta_{1 n}\left(\alpha_{1 n}^2+ \alpha_{2
n}^2-\kappa_2^2\right)\right]\notag\\
&=&-\frac{{\rm i}\mu}{\chi_n}\left[
\left(\beta_{2 n}-\beta_{1 n}
\right)\alpha_{2 n}^2 +\beta_{1 n}\kappa_2^2\right].
\end{eqnarray}
By \eqref{beta}, 
\[
-{\rm
i}\left(\beta_{2 n}-\beta_{1 n}\right)=(
|\boldsymbol{\alpha}_n|^2-\kappa_2^2)^{1/2}
-(|\boldsymbol{\alpha}_n|^2-\kappa_1^2)^{1/2}=\frac{
\kappa_1^2-\kappa_2^2}{(
|\boldsymbol{\alpha}_n|^2-\kappa_2^2)^{1/2}
+(|\boldsymbol{\alpha}_n|^2-\kappa_1^2)^{1/2}}. 
\]
Substituting the above equation into \eqref{M11}, we get 
\begin{eqnarray*}
&&-{\rm i}\left[
\left(\beta_{2 n}-\beta_{1 n}\right)\alpha_{2 n}^2
+\beta_{1 n}\kappa_2^2\right]\\
&&=\frac{\alpha_{2 n}^2\left(\kappa_1^2-\kappa_2^2\right)}{
(|\boldsymbol{\alpha}_n|^2-\kappa_2^2)^{1/2}
+(|\boldsymbol{\alpha}_n|^2-\kappa_1^2)^{1/2}}
+\frac{\kappa_2^2 (|\boldsymbol{\alpha}_n|^2-\kappa_1^2)^{1/2}
\left( (|\boldsymbol{\alpha}_n|^2-\kappa_2^2)^{1/2}
+(|\boldsymbol{\alpha}_n|^2-\kappa_1^2)^{1/2}\right)}
{(|\boldsymbol{\alpha}_n|^2-\kappa_2^2)^{1/2}
+ (|\boldsymbol{\alpha}_n|^2-\kappa_1^2)^{1/2}}\\
&&=\frac{\alpha_{2 n}^2 (\kappa_1^2-\kappa_2^2)+\kappa_2^2
(|\boldsymbol{\alpha}_n|^2-\kappa_1^2)
+\kappa_2^2 (|\boldsymbol{\alpha}_n|^2-\kappa_2^2)^{1/2}
(|\boldsymbol{\alpha}_n|^2-\kappa_1^2)^{1/2}}{(
|\boldsymbol{\alpha}_n|^2-\kappa_2^2)^{1/2}
+ (|\boldsymbol{\alpha}_n|^2-\kappa_1^2)^{1/2}}\\
&&=\frac{\alpha_{2 n}^2\kappa_1^2+ \alpha_{1 n}^2\kappa_2^2
-\kappa_1^2\kappa_2^2+\kappa_2^2 (|\boldsymbol{\alpha}_n|^2-\kappa_2^2)^{1/2}
(|\boldsymbol{\alpha}_n|^2-\kappa_1^2)^{1/2}}{(|\boldsymbol{\alpha}
_n|^2-\kappa_2^2)^{1/2}+ (|\boldsymbol{\alpha}_n|^2-\kappa_1^2)^{1/2}}\\
&&=\frac{\kappa_1^2 (|\boldsymbol{\alpha}_n|^2-\kappa_2^2)
+\alpha_{1 n}^2\left(\kappa_2^2-\kappa_1^2\right)+\kappa_2^2
(|\boldsymbol{\alpha}_n|^2-\kappa_2^2)^{1/2}
(|\boldsymbol{\alpha}_n|^2-\kappa_1^2)^{1/2}}{
(|\boldsymbol{\alpha}_n|^2-\kappa_2^2)^{1/2}
+(|\boldsymbol{\alpha}_n|^2-\kappa_1^2)^{1/2}}
>0,
\end{eqnarray*}
which shows that $\hat{M}_{11}^{(n)}$ is positive. 

The determinant of the leading principal submatrix of order 2 for the
matrix $\hat{M}_n$ is
\begin{eqnarray*}
&&\left(-\frac{{\rm i}\mu}{\chi_n}\right)^2\Big\{
\left[ \alpha_{1 n}^2\left(\beta_{1
n}-\beta_{2 n}\right)+\beta_{2 n}\chi_n\right]
\left[\alpha_{2 n}^2\left(\beta_{1
n}-\beta_{2 n}\right)+\beta_{2 n}\chi_n\right]-\alpha_{1 n}^2 \alpha_{2
n}^2\left(\beta_{1 n}-\beta_{2 n}\right)^2\Big\}\\
&&=-\left(\frac{\mu}{\chi_n}\right)^2\left[|\boldsymbol{
\alpha}_n|^2\left(\beta_{1 n}-\beta_{2 n}\right)
\beta_{2 n}\chi_n+\beta_{2 n}^2 \chi_n^2\right].
\end{eqnarray*}
A simple calculation yields 
\begin{eqnarray*}
&&|\boldsymbol{\alpha}_n|^2\beta_{1n}\beta_{2n}\chi_n
-|\boldsymbol{\alpha}_n|^2\beta_{2n}^2\chi_n+\beta_{2n}^2\chi_n^2\\
&&=|\boldsymbol{\alpha}_n|^2\beta_{1n}\beta_{2n}\chi_n
+\beta_{2n}^2\chi_n\left(|\boldsymbol{\alpha}_n|^2
+\beta_{1n}\beta_{2n}-|\boldsymbol{\alpha}_n|^2\right)\\
&&=|\boldsymbol{\alpha}_n|^2\beta_{1n}\beta_{2n}\chi_n
+\beta_{2n}^2\chi_n\beta_{1n}\beta_{2n}\\
&&=\beta_{1n}\beta_{2n}\chi_n
\left(|\boldsymbol{\alpha}_n|^2
-|\boldsymbol{\alpha}_n|^2+\kappa_2^2\right)
=\beta_{1n}\beta_{2n}\chi_n\kappa_2^2<0.
\end{eqnarray*}
Hence
\[
-\left(\frac{\mu}{\chi_n}\right)^2\left[|\boldsymbol{
\alpha}_n|^2\left(\beta_{1n}-\beta_{2n}\right)\beta_{2n}\chi_n
+\beta_{2n}^2\chi_n^2\right]>0,
\]
which shows that the determinant of the leading principal submatrix of order 2
is also positive. 

It follows from a straightforward calculation that the determinant of matrix
$\hat{M}_n$ itself is
\begin{eqnarray}\label{DetMn}
&&\left(-\frac{{\rm i}\mu}{\chi_n}\right)^3\bigg\{
\big(\alpha_{1n}^2\beta_{12}^{(n)} +\beta_{2n}\chi_n\big)
\Big(\big(\alpha_{2n}^2\beta_{12}^{(n)}
+\beta_{2n}\chi_n\big) \kappa_2^2\beta_{2n}
+\alpha_{2n}^2\beta_{2n}^2 (\beta_{12}^{(n)})^2\Big)\notag \\
&&\quad-\alpha_{1n}\alpha_{2n}\beta_{12}^{(n)}\Big[
\alpha_{1n}\alpha_{2n}\beta_{12}^{(n)}\kappa_2^2\beta_{2n}
+\alpha_{1n}\alpha_{2n}\beta_{2n}^2 (\beta_{12}^{(n)})^2\Big]
+\alpha_{1n}\beta_{2n}(\beta_{12}^{(n)})^2\alpha_{1n} \beta_{2n}^2
\chi_n\bigg\}\notag\\
&=&\left(-\frac{{\rm i}\mu}{\chi_n}\right)^3\Big(
\alpha_{1n}^2\beta_{12}^{(n)}\beta_{2n}^2
\chi_n\kappa_2^2+\beta_{2n}^2\chi_n \alpha_{2n}^2
\beta_{12}^{(n)}\kappa_2^2+\beta_{2n}^3 \chi_n^2\kappa_2^2
+\beta_{2n}^3\chi_n\alpha_{2n}^2(\beta_{12}^{(n)})^2\notag\\
&&\quad+\alpha_{1n}^2\beta_{2n}^3 (\beta_{12}^{(n)})^2\chi_n \Big)\notag\\
&=&\left(-\frac{{\rm i}\mu}{\chi_n}\right)^3\beta_{2n}^2
\Big(\kappa_2^2|\boldsymbol{\alpha}_n|^2\chi_n
\beta_{12}^{(n)}+|\boldsymbol{\alpha}_n|^2
(\beta_{12}^{(n)})^2\beta_{2n}\chi_n
+\beta_{2n}\chi_n^2\kappa_2^2\Big)\notag\\
&=&\left(-\frac{{\rm i}\mu}{\chi_n}\right)^3\beta_{2n}^2
\chi_n\Big(\kappa_2^4\beta_{1n}+|\boldsymbol{\alpha}_n|^2\beta_{2n}
\beta_{1n}^2+|\boldsymbol{\alpha}_n|^2
\beta_{2n}^3-2|\boldsymbol{\alpha}_n|^2
\beta_{1n}\beta_{2n}^2\Big)\notag\\
&=&\left(-\frac{{\rm i}\mu}{\chi_n}\right)^3
\beta_{2n}^2\chi_n\Big(\kappa_2^4\beta_{1n}
-\kappa_2^2|\boldsymbol{\alpha}_n|^2\beta_{12}^{(n)}
+|\boldsymbol{\alpha}_n|^4\beta_{12}^{(n)}+|\boldsymbol{\alpha}
_n|^2\beta_{2n} \beta_{1n}^2-|\boldsymbol{\alpha}_n|^2\beta_{1n}
\beta_{2n}^2\Big)\notag\\
&=&\left(-\frac{{\rm i}\mu}{\chi_n}\right)^3\beta_{2n}^2\chi_n\Big(
\kappa_2^2|\boldsymbol{\alpha}_n|^2\beta_{2n}
+\kappa_2^2\beta_{1n}\beta_{2n}^2
+|\boldsymbol{\alpha}_n|^2\left(\beta_{1n}
-\beta_{2n}\right)\chi_n\Big)\notag\\
&=&\left(-\frac{{\rm i}\mu}{\chi_n}\right)^3
\beta_{2n}^2\chi_n\Big(
\kappa_2^2\beta_{2n}\chi_n+|\boldsymbol{
\alpha}_n|^2\left(\beta_{1n}-
\beta_{2n}\right)\chi_n\Big).
\end{eqnarray}
Recall $\beta_{2n}={\rm i}(|\boldsymbol{\alpha}_n|^2-\kappa_2^2)^{1/2}$. 
The first part of \eqref{DetMn} satisfies
\begin{eqnarray*}
\left(-\frac{{\rm i}\mu}{\chi_n}\right)^3\beta_{2n}^2
\chi_n\kappa_2^2\beta_{2n}\chi_n&=&-\frac{\mu^3}{\chi_n}\kappa_2^2
{\rm i}\big(|\boldsymbol{\alpha}_n|^2-\kappa_2^2\big){\rm
i} \big(|\boldsymbol{\alpha}_n|^2-\kappa_2^2\big)^{1/2} \\
&=&\frac{\mu^3}{\chi_n}\kappa_2^2
\big(|\boldsymbol{\alpha}_n|^2-\kappa_2^2\big)^{
3/2}>0.
\end{eqnarray*}
The second part of \eqref{DetMn} is 
\begin{eqnarray*}
&&\left(-\frac{{\rm i}\mu}{\chi_n}\right)^3 \beta_{2n}^2\chi_n
|\boldsymbol{\alpha}_n|^2\left(\beta_{1n}-\beta_{2n}\right)\chi_n
\\
&&=\frac{\mu^3}{\chi_n}|\boldsymbol{\alpha}_n|^2 (
|\boldsymbol{\alpha}_n|^2-\kappa_2^2)
\big( (|\boldsymbol{\alpha}_n|^2-\kappa_1^2)^{1/2}
-(|\boldsymbol{\alpha}_n|^2-\kappa_2^2)^{1/2}\big).
\end{eqnarray*}
Since $\kappa_1<\kappa_2$, we have 
$(|\boldsymbol{\alpha}_n|^2-\kappa_1^2)^{1/2}
-(|\boldsymbol{\alpha}_n|^2-\kappa_2^2)^{1/2}>0$. The proof is completed after
combining the above estimates.
\end{proof}

\begin{lemma}\label{lemma5}
Let $\Omega'=\big\{\boldsymbol{x}\in\mathbb{R}^3: (x_1,
x_2)\in(0, \Lambda_1)\times(0, \Lambda_2), \hat h<x_3<h\big\}$. Then for any
$\delta>0$, there exists a positive constant $C(\delta)$ independent of $N$ such
that
\[
\Re\int_{\Gamma_h}T_N\boldsymbol{\xi}\cdot\overline{
\boldsymbol{\xi}}\,{\rm d}s\leq
C(\delta)\|\boldsymbol{\xi}\|_{\boldsymbol{L}^2(\Omega')}
^2+\delta\|\boldsymbol{\xi}\|_{\boldsymbol{H}^1(\Omega')}^2.
\]
\end{lemma}

\begin{proof}
It follows from the definition of the DtN operator \eqref{DtN} that we have
\begin{eqnarray*}
\Re\int_{\Gamma_{h}}T_N\boldsymbol{\xi}\cdot\overline{
\boldsymbol{\xi}}\,{\rm d}s &=&
\Lambda_1\Lambda_2\Re\sum\limits_{|n_1|, |n_2|\leq
N}(M_n\boldsymbol{\xi}_n)
\cdot\overline{\boldsymbol{\xi}_n}\\
&=&-\Lambda_1\Lambda_2\sum\limits_{|n_1|, |n_2|\leq
N}(\hat{M}_n\boldsymbol{\xi}_n)
\cdot\overline{\boldsymbol{\xi}_n}.
\end{eqnarray*}
By Lemma \ref{lemma4}, $\hat{M}_n$ is positive definite for sufficiently large
$|n|_{\max}$. Hence for fixed $\omega, \lambda, \mu$, there exists a positive
integer $N^*$ such that 
\begin{eqnarray*}
\Re\int_{\Gamma_{h}}T_N\boldsymbol{\xi}\cdot\overline{
\boldsymbol{\xi}}\,{\rm d}s\leq
-\Lambda_1\Lambda_2\sum\limits_{|n|_{\max}\leq \min(N,
N^*)}(\hat{M}_n\boldsymbol{\xi}_n)\cdot\overline{
\boldsymbol{\xi}_n}\quad\forall \,|n|_{\max}>N^*.
\end{eqnarray*}
On the other hand, there exists a constant $C$ depending only on $\omega, \mu,
\lambda$ such that
\[
\big|(\hat{M}_n\boldsymbol{\xi}_n)\cdot\overline{\boldsymbol{\xi}_n}\big|\leq
C|\boldsymbol{\xi}_n|^2\quad\forall\,|n|_{\max}\leq \min(N^*, N).
\]

For any $\delta>0$, it follows from Young's inequality that
\begin{eqnarray*}
&&(h-\hat h)\left|\phi(h)\right|^2=\int_{\hat h}^{h}
\left|\phi(x_3)\right|^2 {\rm d}x_3+
\int_{\hat h}^{h}\int_{x_3}^{h}\big(\left|\phi(s)\right|^2\big)' 
{\rm d}s{\rm d}x_3\\
&&\leq \int_{\hat h}^{h}\left|\phi(x_3)\right|^2{\rm
d}x_3+\Big(\frac{h-\hat h}{\delta}\Big)
\int_{\hat h}^{h}\left|\phi(x_3)\right|^2{\rm
d}x_3+\delta(h-\hat h)\int_{\hat h}^{h}\left|\phi^\prime(x_3)\right|^2{\rm
d}x_3,
\end{eqnarray*}
which gives that
\begin{eqnarray*}
\left|\phi(h)\right|^2\leq
\left[\delta^{-1}+(h-\hat h)^{-1}\right]\int_{\hat h}^{h}
\left|\phi(x_3)\right|^2{\rm d}x_3
+\delta\int_{\hat h}^{h}\left|\phi^\prime(x_3)\right|^2{\rm
d}x_3.
\end{eqnarray*}
Let
$\phi(\boldsymbol
x)=\sum\limits_{n\in\mathbb{Z}^2}\phi_n(x_3)e^{{\rm i}
\boldsymbol{\alpha}_n\cdot\boldsymbol{r}}$. It is easy to get 
\begin{eqnarray*}
&&\|\nabla\phi\|^2_{\boldsymbol{L}^2(\Omega')}
=\Lambda_1\Lambda_2\sum\limits_{n\in\mathbb{Z}^2}
\int_{\hat h}^{h}\left(|\phi_n'
(x_3)|^2+|\boldsymbol{\alpha}_n|^2
|\phi_n(x_3)|^2\right){\rm d}x_3,\\
&&\|\phi\|^2_{L^2(\Omega')}=\Lambda_1\Lambda_2\sum\limits_{
n\in\mathbb{Z}^2}
\int_{\hat h}^{h}|\phi_n(x_3)|^2{\rm d}x_3.
\end{eqnarray*}
Hence, we have for any $\phi\in
\boldsymbol{H}^{1}(\Omega')$ that 
\begin{eqnarray*}
\|\phi\|^2_{L^2(\Gamma_{h})} &=& \Lambda_1\Lambda_2
\sum\limits_{n\in\mathbb{Z}^2}|\phi_n(h)|^2\\
&\leq& \Lambda_1\Lambda_2\left[\delta^{-1}+(h-\hat h)^{-1}\right]
\sum\limits_{n\in\mathbb{Z}^2}\int_{\hat h}^{h}|\phi_n(x_3)|^2{\rm
d}x_3+\Lambda_1\Lambda_2\delta\sum\limits_{n\in\mathbb{Z}^2}
\int_{\hat h}^{h}|\phi_n'(x_3)|^2{\rm d}x_3\\
&\leq& \Lambda_1\Lambda_2\left[\frac{1}{\delta}+(h_1-h_2)^{-1}\right]
\sum\limits_{n\in\mathbb{Z}^2}\int_{\hat h}^{h}|\phi_n(x_3)|^2{\rm
d}x_3\\
&& \quad +\Lambda_1\Lambda_2\delta\sum\limits_{n\in\mathbb{Z}
^2}\int_{\hat h}^{h}\left[|\phi_n'(x_3)|^2+
|\boldsymbol{\alpha}_n|^2|\phi_n(x_3)|^2\right]{\rm d}x_3\\
&\leq& \left[\delta^{-1}+(h-\hat h)^{-1}\right]\|\phi\|^2_{L^2(\Omega')}
+\delta\|\nabla\phi\|^2_{\boldsymbol{L} ^2(\Omega')}\\
&\leq& C(\delta)\|\phi\|^2_{L^2(\Omega')}+\delta\|\nabla\phi\|^2_{
\boldsymbol{L}^2(\Omega')}.
\end{eqnarray*}
Combining the above estimates, we obtain 
\begin{eqnarray*}
\Re\int_{\Gamma_{h}}T_N\boldsymbol{\xi}\cdot\overline{
\boldsymbol{\xi}}{\rm d}s\leq C
\|\boldsymbol{\xi}\|^2_{\boldsymbol{L}^2(\Gamma_{h})} 
\leq C(\delta)\|\boldsymbol\xi\|^2_{\boldsymbol{L}^2(\Omega')}
+\delta\|\boldsymbol\xi\|^2_{ \boldsymbol{H} ^1(\Omega')},
\end{eqnarray*}
which completes the proof. 
\end{proof}

To estimate the third term of \eqref{Error}, we introduce the dual problem
\begin{equation}\label{DualProblem}
a(\boldsymbol{v},\boldsymbol{p})=\int_{\Omega}\boldsymbol{v}\cdot\overline{
\boldsymbol{\xi}}\,{\rm d}\boldsymbol{x}
\quad \forall\, \boldsymbol{v}\in H_{S,{\rm qp}}^{1}(\Omega).
\end{equation}
It is easy to check that $\boldsymbol{p}$ is the weak solution of the boundary
value problem
\begin{equation}\label{DualProblem2}
\begin{cases}
\mu\Delta\boldsymbol{p}+(\lambda+\mu)\nabla\nabla\cdot\boldsymbol{p}
+\omega^2\boldsymbol{p}=-\boldsymbol{\xi}
\quad &{\rm in}\,\Omega\\
\boldsymbol{p}=0	&{\rm on}\,S\\
B\boldsymbol{p}=T^{*}\boldsymbol{p} &{\rm on}\,\Gamma_{h}
\end{cases}
\end{equation}
where $T^*$ is the adjoint operator to $T$ under the scalar product in
$L^2(\Gamma_h)$. Taking
$\boldsymbol{v}=\boldsymbol{\xi}$ in \eqref{DualProblem}, we have 
\begin{equation}\label{xiL2}
\|\boldsymbol{\xi}\|^2_{\boldsymbol{L}^2(\Omega)}=a(\boldsymbol{\xi},
\boldsymbol{p}) -\int_{\Gamma_{h}}\left(T-T_N\right)\boldsymbol{\xi}
\cdot\overline{\boldsymbol{p}}\,{\rm d}s
+\int_{\Gamma_{h}}\left(T-T_N\right)\boldsymbol{\xi}
\cdot\overline{\boldsymbol{p}}\,{\rm d}s.
\end{equation}

It is clear that the evaluation of $\boldsymbol{p}$ is essential to the error
estimate. Lemmas \ref{lemmaZ}--\ref{lemmaP} give the asymptotic analysis of
$\boldsymbol{p}$. First, we introduce the Helmholtz decomposition of
$\boldsymbol{\xi}$ in $\Omega'$: 
\begin{equation}\label{HelmholtzXi}
\boldsymbol{\xi}=\nabla\zeta+\nabla\times \boldsymbol{Z},\quad
\nabla\cdot \boldsymbol{Z}=0,
\end{equation}
where $ \boldsymbol{Z}=(Z_1, Z_2, Z_3)^\top$ and 
\begin{eqnarray*}
\zeta(\boldsymbol{x})=\sum\limits_{\boldsymbol{n}\in\mathbb{Z}^2}\zeta_n(x_3)e^{
{\rm i}\boldsymbol{\alpha}_n\cdot\boldsymbol{r}}, \quad
Z_j(\boldsymbol{x})=\sum\limits_{\boldsymbol{n}\in\mathbb{Z}^2}Z_{jn}
(x_3)e^{{\rm i}\boldsymbol{\alpha}_n\cdot\boldsymbol{r}}.
\end{eqnarray*}
Substituting the above Fourier series expansions into \eqref{HelmholtzXi}
gives 
\begin{equation}\label{systemxi}
\begin{bmatrix}
Z_{1n}'(x_3) \\
Z_{2n}'(x_3)\\
Z_{3n}'(x_3)\\
\zeta_n'(x_3)
\end{bmatrix}=
\begin{bmatrix}
0 & 0 & {\rm i}\alpha_{1n} & -{\rm i}\alpha_{2n}\\
0 & 0 & {\rm i}\alpha_{2n} &  {\rm i}\alpha_{1n}\\
-{\rm i}\alpha_{1n} & -{\rm i}\alpha_{2n} & 0 & 0\\
{\rm i}\alpha_{2n} & -{\rm i}\alpha_{1n} & 0 & 0 
\end{bmatrix}
\begin{bmatrix}
Z_{1n}(x_3) \\
Z_{2n}(x_3) \\
Z_{3n}(x_3)\\
\zeta_n(x_3)
\end{bmatrix}
+
\begin{bmatrix}
\xi_{2n}(x_3) \\ 
-\xi_{1n}(x_3) \\ 0 \\
\xi_{3n}(x_3)
\end{bmatrix}.
\end{equation}
In addition, the homogeneous Dirichlet boundary condition is imposed 
for the Fourier coefficients at $x_3=h$:
\begin{equation}\label{BoundaryZ}
Z_{1n}(h)=Z_{2n}(h)=Z_{3n}(h)=\zeta_n(h)=0.
\end{equation}

\begin{lemma}\label{lemmaZ}
The solutions of the problem \eqref{systemxi}--\eqref{BoundaryZ} in $[\hat h,
h]$ satisfy the following estimates:
\begin{eqnarray*}
\left|\zeta_n(x_3)\right|&\lesssim&
\|\boldsymbol{\xi}_n\|_{\boldsymbol{L}^{\infty}([\hat h, h])}
\frac{1}{|\boldsymbol{\alpha}_n|}
e^{|\boldsymbol{\alpha}_n|(h-x_3)},\\
\left|Z_{jn}(x_3)\right|&\lesssim&
\|\boldsymbol{\xi}_n\|_{\boldsymbol{L}^{\infty}([\hat h, h])}
\frac{1}{|\boldsymbol{\alpha}_n|}
e^{|\boldsymbol{\alpha}_n|(h-x_3)}, \quad
j=1,2,3.\\
\end{eqnarray*}	
\end{lemma}

\begin{proof}
Denote the coefficient matrix of system \eqref{systemxi} by
$A_n$, which is the matrix $A_n$ defined in \eqref{uphi}. A
straightforward calculation shows that the
coefficient matrix $A_n$ has the following diagonalization: 
\[
A_n=V_n\begin{bmatrix}
|\boldsymbol{\alpha}_n| & 0 & 0 & 0\\
0 & |\boldsymbol{\alpha}_n| & 0 & 0\\
0 & 0 & -|\boldsymbol{\alpha}_n| & 0\\
0 & 0 & 0 & -|\boldsymbol{\alpha}_n|
\end{bmatrix}V_n^*,
\]
where 
\[
V_n=\frac{1}{\sqrt{2}|\boldsymbol{\alpha}_n|}
\begin{bmatrix}
|\boldsymbol{\alpha}_n| & 0 &
|\boldsymbol{\alpha}_n| & 0\\
0 & |\boldsymbol{\alpha}_n| & 0 &
|\boldsymbol{\alpha}_n|\\
-{\rm i}\alpha_{1n} & -{\rm i}\alpha_{2n} & {\rm
i}\alpha_{1n} & {\rm i}\alpha_{2n}\\
{\rm i}\alpha_{2n} & -{\rm i}\alpha_{1n} & -{\rm
i}\alpha_{2n} & {\rm i}\alpha_{1n}
\end{bmatrix}.
\]
Hence the fundamental solution of \eqref{systemxi} is
\begin{eqnarray*}
\Phi_n(x_3) &=& e^{\int_{\hat h}^{x_3}
A_n(\tau) \,d\tau} \\
&=& V_n\begin{bmatrix}
e^{|\boldsymbol{\alpha}_n|\left(x_3-\hat h\right)} & 0 & 0 & 0\\
0 &
e^{|\boldsymbol{\alpha}_n|\left(x_3-\hat h\right)} & 0
& 0\\
0 & 0 &
e^{-|\boldsymbol{\alpha}_n|\left(x_3-\hat h\right)} &
0\\
0 & 0 & 0 &
e^{-|\boldsymbol{\alpha}_n|\left(x_3-\hat h\right)}
\end{bmatrix}V_n^*\\
&=& \frac{1}{2|\boldsymbol{\alpha}_n|}
\left(D_{1n}+D_{2n}\right),
\end{eqnarray*}
where
\begin{equation*}
D_{1n}=
\begin{bmatrix}
|\boldsymbol{\alpha}_n|e^{|\boldsymbol{\alpha}_n|\left(x_3-\hat h\right)} &
0 &{\rm i}\alpha_{1n}e^{|\boldsymbol{\alpha}_n|\left(x_3-h_2\right)} &
-{\rm i}\alpha_{2n}e^{|\boldsymbol{\alpha}_n|\left(x_3-\hat h\right)}\\
0 & |\boldsymbol{\alpha}_n|e^{|\boldsymbol{\alpha}_n|\left(x_3-h_2\right)} &
{\rm i}\alpha_{2n}e^{|\boldsymbol{\alpha}_n|\left(x_3-\hat h\right)} &
{\rm i}\alpha_{1n}e^{|\boldsymbol{\alpha}_n|\left(x_3-\hat h\right)}\\
-{\rm i}\alpha_{1n}e^{|\boldsymbol{\alpha}_n|\left(x_3-\hat h\right)} &
-{\rm i}\alpha_{2n}e^{|\boldsymbol{\alpha}_n|\left(x_3-\hat h\right)} &
|\boldsymbol{\alpha}_n|e^{|\boldsymbol{\alpha}_n|\left(x_3-\hat h\right)} &
0\\ {\rm i}\alpha_{2n}e^{|\boldsymbol{\alpha}_n|\left(x_3-\hat h\right)} &
-{\rm i}\alpha_{1n}e^{|\boldsymbol{\alpha}_n|\left(x_3-\hat h\right)} &
0 &|\boldsymbol{\alpha}_n|e^{|\boldsymbol{\alpha}_n|\left(x_3-\hat h\right)} 
\end{bmatrix}
\end{equation*}	 and
\begin{equation*}
D_{2n}=
\begin{bmatrix}
|\boldsymbol{\alpha}_n|e^{-|\boldsymbol{\alpha}_n |\left(x_3-\hat h\right)} &
0 &-{\rm i}\alpha_{1n}e^{-|\boldsymbol{\alpha}_n|\left(x_3-\hat h\right)} &
{\rm i}\alpha_{2n}e^{-|\boldsymbol{\alpha}_n|\left(x_3-\hat h\right)}\\
0 & |\boldsymbol{\alpha}_n|e^{-|\boldsymbol{\alpha}_n|\left(x_3-\hat h\right)} &
-{\rm i}\alpha_{2n}e^{-|\boldsymbol{\alpha}_n|\left(x_3-\hat h\right)} &
-{\rm i}\alpha_{1n}e^{-|\boldsymbol{\alpha}_n|\left(x_3-\hat h\right)}\\
{\rm i}\alpha_{1n}e^{-|\boldsymbol{\alpha}_n|\left(x_3-\hat h\right)} &
{\rm i}\alpha_{2n}e^{-|\boldsymbol{\alpha}_n|\left(x_3-\hat h\right)} &
|\boldsymbol{\alpha}_n|e^{-|\boldsymbol{\alpha}_n|\left(x_3-\hat h\right)} &
0\\ -{\rm i}\alpha_{2n}e^{-|\boldsymbol{\alpha}_n|\left(x_3-\hat h\right)} &
{\rm i}\alpha_{1n}e^{-|\boldsymbol{\alpha}_n|\left(x_3-\hat h\right)} &
0 &|\boldsymbol{\alpha}_n|e^{-|\boldsymbol{\alpha}_n|\left(x_3-\hat h\right)}
\end{bmatrix}. 
\end{equation*}
The inverse of the fundamental matrix is
\[
\Phi_n^{-1}(x_3) =
\frac{1}{2|\boldsymbol{\alpha}_n|}(\hat{D}_{1n}+\hat{D}_{2n}),
\]
where
\begin{equation*}
\hat{D}_{1n}=
\begin{bmatrix}
|\boldsymbol{\alpha}_n|e^{-|\boldsymbol{\alpha}_n
|\left(x_3-\hat h\right)} &
0 &{\rm
i}\alpha_{1n}e^{-|\boldsymbol{\alpha}_n|\left(x_3-\hat h\right)} &
-{\rm i}\alpha_{2n}e^{-|\boldsymbol{\alpha}_n|\left(x_3-\hat h\right)}\\
0 & |\boldsymbol{\alpha}_n|e^{-|\boldsymbol{\alpha}_n|\left(x_3-\hat h\right)}
& {\rm i}\alpha_{2n}e^{-|\boldsymbol{\alpha}_n|\left(x_3-\hat h\right)} &
{\rm i}\alpha_{1n}e^{-|\boldsymbol{\alpha}_n|\left(x_3-\hat h\right)}\\
-{\rm i}\alpha_{1n}e^{-|\boldsymbol{\alpha}_n|\left(x_3-\hat h\right)} &
-{\rm i}\alpha_{2n}e^{-|\boldsymbol{\alpha}_n|\left(x_3-\hat h\right)} &
|\boldsymbol{\alpha}_n|e^{-|\boldsymbol{\alpha}_n|\left(x_3-\hat h\right)} &
0\\ {\rm i}\alpha_{2n}e^{-|\boldsymbol{\alpha}_n|\left(x_3-\hat h\right)} &
-{\rm i}\alpha_{1n}e^{-|\boldsymbol{\alpha}_n|\left(x_3-\hat h\right)} &
0 &|\boldsymbol{\alpha}_n|e^{-|\boldsymbol{\alpha}_n|\left(x_3-\hat h\right)} 
\end{bmatrix}
\end{equation*}
and
\begin{equation*}
\hat{D}_{2n}=
\begin{bmatrix}
|\boldsymbol{\alpha}_n|e^{|\boldsymbol{\alpha}_n|\left(x_3-\hat h\right)} &
0 &-{\rm i}\alpha_{1n}e^{|\boldsymbol{\alpha}_n|\left(x_3-\hat h\right)} &
{\rm i}\alpha_{2n}e^{|\boldsymbol{\alpha}_n|\left(x_3-\hat h\right)}\\
0 & |\boldsymbol{\alpha}_n|e^{|\boldsymbol{\alpha}_n|\left(x_3-\hat h\right)}
& -{\rm i}\alpha_{2n}e^{|\boldsymbol{\alpha}_n|\left(x_3-\hat h\right)} &
-{\rm i}\alpha_{1n}e^{|\boldsymbol{\alpha}_n|\left(x_3-\hat h\right)}\\
{\rm i}\alpha_{1n}e^{|\boldsymbol{\alpha}_n|\left(x_3-\hat h\right)} &
{\rm i}\alpha_{2n}e^{|\boldsymbol{\alpha}_n|\left(x_3-\hat h\right)} &
|\boldsymbol{\alpha}_n|e^{|\boldsymbol{\alpha}_n|\left(x_3-\hat h\right)} &
0\\ -{\rm i}\alpha_{2n}e^{|\boldsymbol{\alpha}_n|\left(x_3-\hat h\right)} &
{\rm i}\alpha_{1n}e^{|\boldsymbol{\alpha}_n|\left(x_3-\hat h\right)} &
0 &|\boldsymbol{\alpha}_n|e^{|\boldsymbol{\alpha}_n|\left(x_3-\hat h\right)} 
\end{bmatrix}.
\end{equation*}

By the method of the variation of parameters, the solution of \eqref{systemxi}
is
\begin{equation}\label{DcZ}
(Z_{1n}(x_3), Z_{2n}(x_3), Z_{3n}(x_3),
\zeta_n(x_3))^\top=\Phi_n(x_3)C_n(x_3).
\end{equation}
It can be easily verified that the vector of unknowns $C_n=(C_{1n},
C_{2n}, C_{3n}, C_{4n})^\top$ satisfies
\begin{equation*}
C'_n(x_3)=\Phi_n^{-1}(x_3)(\xi_{2n}(x_3), -\xi_{1n}(x_3), 
0, \xi_{3n}(x_3))^\top,
\end{equation*}
which has the solution
\begin{eqnarray*}
C_{1n}(x_3)&=&-\frac{1}{2}\int_{x_3}^{h}\xi_{2n}(t)a_n(t)\,{\rm d}t
-\frac{\rm
i}{2}\frac{\alpha_{2n}}{|\boldsymbol{\alpha}_n|}\int_{x_3}^{h}\xi_{3n}
(t)b_n(t)\, {\rm d}t,\\
C_{2n}(x_3)&=&\frac{1}{2}\int_{x_3}^{h}\xi_{1n}
(t)a_n(t)\,{\rm d}t+\frac{\rm
i}{2}\frac{\alpha_{1n}}{|\boldsymbol{\alpha}_n|}\int_{x_3}^{h}\xi_{3n}(t)
b_n(t)\,{\rm d}t,\\
C_{3n}(x_3)&=&\frac{\rm
i}{2}\frac{\alpha_{2n}}{|\boldsymbol{\alpha}_n|
}\int_{x_3}^{h}\xi_{1n}(t)b_n(t)\,{\rm d}t-\frac{\rm
i}{2}\frac{\alpha_{1n}}{|\boldsymbol{\alpha}_n|
}\int_{x_3}^{h}\xi_{2n}(t)b_n(t)\,{\rm d}t,\\	
C_{4n}(x_3)&=&\frac{\rm
i}{2}\frac{\alpha_{2n}}{|\boldsymbol{\alpha}_n|
}\int_{x_3}^{h}\xi_{2n}(t)b_n(t)\,{\rm d}t+\frac{\rm
i}{2}\frac{\alpha_{1n}}{|\boldsymbol{\alpha}_n|
}\int_{x_3}^{h}\xi_{1n}(t)b_n(t)\,{\rm
d}t-\frac{1}{2}\int_{x_3}^{h}\xi_{3n}(t)a_n(t)\,{\rm d}t.
\end{eqnarray*}
Here
\begin{eqnarray*}
a_n(x_3)=e^{|\boldsymbol{\alpha}_n|\left(x_3-\hat h\right)}
+e^{-|\boldsymbol{\alpha}_n|\left(x_3-\hat h\right)},\quad 
b_n(x_3)=e^{|\boldsymbol{\alpha}_n|\left(x_3-\hat h\right)}
-e^{-|\boldsymbol{\alpha}_n|\left(x_3-\hat h\right)}. 
\end{eqnarray*}
Substituting the expressions of $C_n$ into \eqref{DcZ}, we obtain 
\begin{eqnarray}
&&Z_{1n}(x_3)=-\frac{1}{2}e^{|\boldsymbol{\alpha}_n|x_3}\int_{x_3}^{h}
e^{-|\boldsymbol{\alpha}_n|t}\xi_{2n}(t)\,{\rm d}t
-\frac{1}{2}e^{-|\boldsymbol{\alpha}_n|x_3}\int_{x_3}^{h}
e^{|\boldsymbol{\alpha}_n|t}\xi_{2n}(t)\,{\rm d}t \notag\\
&&\qquad+\frac{\rm
i}{2}\frac{\alpha_{2n}}{|\boldsymbol{\alpha}_n|}e^{|\boldsymbol{\alpha}_n|x_3}
\int_{x_3}^{h}e^{-|\boldsymbol{\alpha}_n|t}\xi_{3n}(t)\,{\rm d}t
-\frac{\rm
i}{2}\frac{\alpha_{2n}}{|\boldsymbol{\alpha}_n|}e^{-|\boldsymbol{\alpha}_n|x_3}
\int_{x_3}^{h}e^{|\boldsymbol{\alpha}_n|t}\xi_{3n}(t)\,{\rm d}t, \label{Z1}\\
&&
Z_{2n}(x_3)=\frac{1}{2}e^{|\boldsymbol{\alpha}_n|x_3}\int_{x_3}^{h}
e^{-|\boldsymbol{\alpha}_n|t}\xi_{1n}(t)\,{\rm d}t
+\frac{1}{2}e^{-|\boldsymbol{\alpha}_n|x_3}\int_{x_3
}^{h}e^{|\boldsymbol{\alpha}_n|t}\xi_{1n}(t)\,{\rm d}t \notag\\
&&\qquad -\frac{\rm i}{2}\frac{\alpha_{1n}}{|\boldsymbol{\alpha}_n|
}e^{|\boldsymbol{\alpha}_n|x_3}\int_{x_3}^{h}e^{-|\boldsymbol{\alpha}_n|t}\xi_{
3n}(t)\,{\rm d}t +\frac{\rm
i}{2}\frac{\alpha_{1n}}{|\boldsymbol{\alpha}_n|}e^{-|\boldsymbol{\alpha}_n|x_3}
\int_{x_3}^{h}e^{|\boldsymbol{\alpha}_n|t}\xi_{3n}(t)\,{\rm d}t, \label{Z2}\\
&&Z_{3n}(x_3)=\frac{\rm
i}{2}\frac{\alpha_{1n}}{|\boldsymbol{\alpha}_n|}e^{|\boldsymbol{\alpha}_n|x_3}
\int_{x_3}^{h}e^{-|\boldsymbol{\alpha}_n|t}\xi_{2n}(t)\,{\rm d}t-\frac{\rm
i}{2}\frac{\alpha_{1n}}{|\boldsymbol{\alpha}_n|}e^{-|\boldsymbol{\alpha}_n|x_3}
\int_{x_3}^{h}e^{|\boldsymbol{\alpha}_n|t}\xi_{2n}(t)\,{\rm d}t\notag\\
&&\qquad -\frac{\rm i}{2}\frac{\alpha_{2n}}{|\boldsymbol{\alpha}_n|
}e^{|\boldsymbol{\alpha}_n|x_3}\int_{x_3}^{h}
e^{-|\boldsymbol{\alpha}_n|t}\xi_{1n}(t)\,{\rm d}t +\frac{\rm
i}{2}\frac{\alpha_{2n}}{|\boldsymbol{\alpha}_n|
}e^{-|\boldsymbol{\alpha}_n|x_3}\int_{x_3}^{h}
e^{|\boldsymbol{\alpha}_n|t}\xi_{1n}(t)\,{\rm d}t \label{Z3}
\end{eqnarray}
and
\begin{eqnarray}
&&\zeta_n(x_3)=-\frac{1}{2}e^{|\boldsymbol{\alpha}_n|x_3}\int_{x_3}^{h}
e^{-|\boldsymbol{\alpha}_n|t}\xi_{3n}(t)\,{\rm d}t
-\frac{1}{2}e^{-|\boldsymbol{\alpha}_n|x_3}\int_{x_3
}^{h}e^{|\boldsymbol{\alpha}_n|t}\xi_{3n}(t)\,{\rm d}t \notag\\
&&\qquad -\frac{\rm
i}{2}\frac{\alpha_{2n}}{|\boldsymbol{\alpha}_n|
}e^{|\boldsymbol{\alpha}_n|x_3}\int_{x_3}^{h}
e^{-|\boldsymbol{\alpha}_n|t}\xi_{2n}(t)\,{\rm d}t
+\frac{\rm
i}{2}\frac{\alpha_{2n}}{|\boldsymbol{\alpha}_n|
}e^{-|\boldsymbol{\alpha}_n|x_3}\int_{x_3}^{h}
e^{|\boldsymbol{\alpha}_n|t}\xi_{2n}(t)\,{\rm d}t\notag\\
&&\qquad-\frac{\rm
i}{2}\frac{\alpha_{1n}}{|\boldsymbol{\alpha}_n|
}e^{|\boldsymbol{\alpha}_n|x_3}\int_{x_3}^{h}
e^{-|\boldsymbol{\alpha}_n|t}\xi_{1n}(t)\,{\rm d}t
+\frac{\rm
i}{2}\frac{\alpha_{1n}}{|\boldsymbol{\alpha}_n|
}e^{-|\boldsymbol{\alpha}_n|x_3}\int_{x_3}^{h}
e^{|\boldsymbol{\alpha}_n|t}\xi_{1n}(t)\,{\rm d}t. \label{zeta}
\end{eqnarray}
It is easy to check from \eqref{Z1}--\eqref{zeta} that 
\begin{eqnarray*}
|\zeta_n(x_3)| &\lesssim&
\|\boldsymbol{\xi}_n\|_{\boldsymbol{L}^{\infty}([\hat h, h])}
\frac{1}{|\boldsymbol{\alpha}_n|}e^{|\boldsymbol{\alpha}_n|(h-x_3)},\\
|Z_{jn}(x_3)|&\lesssim&
\|\boldsymbol{\xi}_n\|_{\boldsymbol{L}^{\infty}([\hat h, h])}
\frac{1}{|\boldsymbol{\alpha}_n|}
e^{|\boldsymbol{\alpha}_n|(h-x_3)}, \quad
j=1,2,3,
\end{eqnarray*}
which complete the proof.
\end{proof}

Consider the following boundary value problem for $\boldsymbol{p}$ in
$\Omega'$: 
\begin{equation}\label{DualProblem3}
\begin{cases}
\mu\Delta\boldsymbol{p}+(\lambda+\mu)\nabla\nabla\cdot\boldsymbol{p}
+\omega^2\boldsymbol{p}=-\boldsymbol{\xi}\quad &\text{in} ~ \Omega',\\
\boldsymbol{p}=\boldsymbol{p} \quad &\text{on} ~ \Gamma_{\hat h},\\
B\boldsymbol{p}=T^{*}\boldsymbol{p} \quad &\text{on} ~ \Gamma_{h}. 
\end{cases}
\end{equation}

\begin{lemma}\label{lemma6}
Let $\boldsymbol{q}=(q_1, q_2, q_3)^\top$ and $g$ have the Fourier
series expansions 
\begin{eqnarray*}
q_j(\boldsymbol
x)=\sum\limits_{n\in\mathbb{Z}^2}q_{jn} (x_3)
e^{{\rm i}\boldsymbol{\alpha}_n\cdot\boldsymbol{r}},\quad
g(\boldsymbol x)=\sum\limits_{n\in\mathbb{Z}^2}g_n
(x_3)e^{{\rm i}\boldsymbol{\alpha}_n\cdot\boldsymbol{r}},\quad\boldsymbol
x\in\Omega'
\end{eqnarray*}
and satisfy
\begin{equation}\label{gqsystem}
\begin{cases}
(\lambda+2\mu)\left(\Delta g+\kappa_1^2 g\right)=-\zeta &\quad {\rm
in} ~ \Omega',\\
\mu\left(\nabla\times(\nabla\times
\boldsymbol{q})-\kappa_2^2\boldsymbol{q}\right)=\boldsymbol{Z},
\quad \nabla\cdot\boldsymbol{q}=0 &\quad {\rm in} ~ \Omega',\\
\boldsymbol{q}=\boldsymbol{q}, \quad g=g &\quad {\rm on} ~ \Gamma_{\hat h}. 
\end{cases}
\end{equation}
Moreover, the Fourier coefficients are assumed to satisfy the following
boundary conditions on $\Gamma_h$: 
\begin{eqnarray}\label{tbcgq1}
g_n'(h)=-{\rm i}\overline{\beta_{1n}}g_n(h)
\end{eqnarray}
and
\begin{eqnarray}\label{tbcgq2}
q_{1n}'(h)=-{\rm
i}\overline{\beta_{2n}}q_{1n}(h),\quad q_{2n}'(h)=-{\rm
i}\overline{\beta_{2n}}q_{2n}(h),\quad
q_{3n}'(h)={\rm
i}\alpha_{1n}q_{1n}(h)+{\rm
i}\alpha_{2n}q_{2n}(h).
\end{eqnarray}
Then $\boldsymbol p$ has the Helmholtz decomposition $\boldsymbol{p}=\nabla
g+\nabla\times\boldsymbol{q}$ and satisfies the boundary
value problem \eqref{DualProblem3}.
\end{lemma}

\begin{proof}
Substituting $\boldsymbol{p}=\nabla g+\nabla\times\boldsymbol{q}$ into the
elastic wave equation, we obtain
\begin{eqnarray*}
&&\mu\Delta\left(\nabla
g+\nabla\times\boldsymbol{q}\right)+(\lambda+\mu)\nabla\nabla\cdot\left(\nabla
g+\nabla\times\boldsymbol{q}\right)
+\omega^2\left(\nabla g+\nabla\times\boldsymbol{q}\right)\\
&&=\nabla\left(\mu\Delta g+(\lambda+\mu)\Delta g+\omega^2
g\right)+\nabla\times\left(\mu\Delta
\boldsymbol{q}+\omega^2\boldsymbol{q}\right)\\
&&=(\lambda+2\mu)\nabla\left(\Delta g+\kappa_1^2
g\right)+\mu\nabla\times\left(-\nabla\times(\nabla\times \boldsymbol{q})
+\kappa_2^2\boldsymbol{q}\right)\\
&&=-\nabla\zeta-\nabla\times\boldsymbol{Z}=-\boldsymbol{\xi}.
\end{eqnarray*}
Since $g(\boldsymbol
x)=\sum\limits_{n\in\mathbb{Z}^2}g_n(x_3)e^{{\rm
i}\boldsymbol{\alpha}_n\cdot\boldsymbol{r}}$, we get from taking the second
order partial derivatives of $g$ that 
\begin{eqnarray*}
\partial_{x_1}^2 g(\boldsymbol
x)&=&-\sum\limits_{n\in\mathbb{Z}}\alpha_{1n}^2 g_n(x_3)e^{{\rm
i}\boldsymbol{\alpha}_n\cdot\boldsymbol{r}},\\
\partial_{x_2}^2 g(\boldsymbol
x)&=&-\sum\limits_{n\in\mathbb{Z}}\alpha_{2n}^2
g_n(x_3)e^{{\rm i}\boldsymbol{\alpha}_n\cdot\boldsymbol{r}},\\
\partial_{x_3}^2 g(\boldsymbol x)&=&\sum\limits_{n\in\mathbb{Z}}
g_n''(x_3)e^{{\rm i}\boldsymbol{\alpha}_n\cdot\boldsymbol{r}}.
\end{eqnarray*}
Substituting the above three expansions into $(\lambda+2\mu)\left(\Delta
g+\kappa_1^2 g\right)=-\zeta$ yields
\begin{eqnarray}
g_n''(x_3)-|\boldsymbol{\alpha}_n|^2 g_n(x_3)
+\kappa_1^2 g_n(x_3)=-\frac{1}{\lambda+2\mu}
\zeta_n(x_3) .\label{g3Z}
\end{eqnarray}
Similarly, we may verify that $q_{jn}$ satisfies the second ordinary
differential equation
\begin{equation}\label{q3Z}
q_{jn}''(x_3)-|\boldsymbol{\alpha}_n|^2 q_{jn}(x_3)
+\kappa_2^2 q_{jn}(x_3)=-\frac{1}{\mu} Z_{jn}(x_3).
\end{equation}

Letting $\boldsymbol{p}=\sum\limits_{n\in\mathbb{Z}^2}
\boldsymbol{p}_n(x_3) e^{{\rm i}\boldsymbol{\alpha}_n\cdot\boldsymbol{r}}$ and 
plugging it into $\boldsymbol{p}=\nabla g+\nabla\times\boldsymbol q$, we get
\begin{equation}\label{valueP}
\begin{bmatrix} p_{1n}(x_3)\\ 
p_{2n}(x_3)\\  p_{3n}(x_3)
\end{bmatrix}=\begin{bmatrix}
{\rm i}\alpha_{1n}g^{(\boldsymbol{n})}(x_3)+{\rm
i}\alpha_{2n}q_{3n}(x_3)-q_{2n}'(x_3)\\
{\rm i}\alpha_{2n}g_n(x_3)-{\rm
i}\alpha_{1n}q_{3n}(x_3)+q_{1n}'(x_3)\\
{\rm i}\alpha_{1n}q_{2n}(x_3)-{\rm
i}\alpha_{2n}q_{1n}(x_3)+g_n'(x_3)
\end{bmatrix}.
\end{equation}
Substituting the above expressions into the boundary operator \eqref{OperatorD}
gives
\begin{eqnarray*}
&& \mu\partial_{x_3}\boldsymbol{p}+(\lambda+\mu)(0,0,1)^\top
\nabla\cdot\boldsymbol{p}\\
&& =\sum\limits_{n\in\mathbb{Z}^2}e^{{\rm
i}\boldsymbol{\alpha}_n\cdot\boldsymbol{r}}\begin{bmatrix}
\mu p_{1n}' \\
\mu p_{2n}' \\
\mu p_{3n}'+(\lambda+\mu){\rm
i}\alpha_{1n}p_{1n}+(\lambda+\mu){\rm i}\alpha_{2n}p_{2n}
+(\lambda+\mu)p_{3n}'
\end{bmatrix}\\
&& =\sum\limits_{n\in\mathbb{Z}^2}e^{{\rm
i}\boldsymbol{\alpha}_n\cdot\boldsymbol{r}}\begin{bmatrix}
\mu \left({\rm i}\alpha_{1n}g_n'+{\rm
i}\alpha_{2n}q_{3n}'-q_{2n}''\right)\\
\mu \left({\rm i}\alpha_{2n}g_n'-{\rm
i}\alpha_{1n}q_{3n}'+q_{1n}''\right)\\
\left(\lambda+2\mu\right)\left({\rm
i}\alpha_{1n}q_{2n}'-{\rm
i}\alpha_{2n}q_{1n}'+g_n''\right)
+\left(\lambda+\mu\right)\left({\rm i}\alpha_{1n}
p_{1n}+{\rm i}\alpha_{2n} p_{2n}\right)
\end{bmatrix}\\
&&=\sum\limits_{n\in\mathbb{Z}^2}e^{{\rm
i}\boldsymbol{\alpha}_n\cdot\boldsymbol{r}}\begin{bmatrix}
{\rm i}\mu\alpha_{1n}g_n'+{\rm
i}\mu\alpha_{2n}q_{3n}'+Z_{2n}
+\mu\left(\kappa_2^2-|\boldsymbol{\alpha}_n|^2\right)q_{2n}\\
{\rm i}\mu\alpha_{2n}g_n'-{\rm
i}\mu\alpha_{1n}q_{3n}'-Z_{1n}
-\mu\left(\kappa_2^2-|\boldsymbol{\alpha}_n|^2\right)q_{1n}\\	
-\zeta_n-(\lambda+2\mu)\left(\kappa_1^2-|\boldsymbol{
\alpha}_n|^2\right)g_n+{\rm i}\mu\alpha_{1n}q_{2n}'-{\rm
i}\mu\alpha_{2n}q_{1n}'-(\lambda+\mu)
|\boldsymbol{\alpha}_n|^2 g_n
\end{bmatrix}.
\end{eqnarray*}
Substituting \eqref{BoundaryZ} and \eqref{tbcgq1}--\eqref{tbcgq2} into the
above equation and evaluating it at $x_3=h$, we get 
\begin{eqnarray*}
&& D\boldsymbol{p}=\mu\partial_{x_3}\boldsymbol{p}+(\lambda+\mu)(0,0,1)^\top
\nabla\cdot\boldsymbol{p}\\
&& =\sum\limits_{n\in\mathbb{Z}^2}e^{{\rm
i}\boldsymbol{\alpha}_n\cdot\boldsymbol{r}}\begin{bmatrix}
-\mu\alpha_{1n}\alpha_{2n} & 
\mu\left(\beta_{2n}^2-\alpha_{2n}^2\right) 
& \mu\alpha_{1n}\overline{\beta_{1n}}\\
-\mu\left(\beta_{2n}^2-\alpha_{2n}^2\right) &
\mu\alpha_{1n}\alpha_{2n} 
& \mu\alpha_{2n}\overline{\beta_{1n}}\\
-\mu\alpha_{2n}\overline{\beta_{2n}} &
\mu\alpha_{1n}\overline{\beta_{2n}} 
& -\mu\kappa_2^2+\mu|\boldsymbol{\alpha}_n|^2
\end{bmatrix}\begin{bmatrix}
q_{1n}(h)\\
q_{2n}(h)\\
g_n(h)
\end{bmatrix}.
\end{eqnarray*}

On the other hand, substituting \eqref{valueP} into \eqref{Mn} gives
\begin{eqnarray*}
&&T^{*}\boldsymbol{p}=\sum\limits_{n\in\mathbb{Z}^2}
M_n^* \boldsymbol{p}_n(h)e^{{\rm
i}\boldsymbol{\alpha}_n\cdot\boldsymbol{r}}\\
&&=\sum\limits_{n\in\mathbb{Z}^2} M_n^*\begin{bmatrix}
0 & {\rm i}\overline{\beta_{2n}} & {\rm
i}\alpha_{2n} & {\rm i}\alpha_{1n}\\
-{\rm i}\overline{\beta_{2n}} & 0 & -{\rm
i}\alpha_{1n} & {\rm i}\alpha_{2n}\\
-{\rm i}\alpha_{2n} & {\rm i}\alpha_{1n} & 0 & -{\rm
i}\overline{\beta_{1n}} 
\end{bmatrix}\begin{bmatrix}
q_{1n}(h)\\
q_{2n}(h)\\
q_{3n}(h)\\
g_n(h)
\end{bmatrix}=\sum\limits_{n\in\mathbb{Z}^2}
K_n\begin{bmatrix}
q_{1n}(h)\\
q_{2n}(h)\\
q_{3n}(h)\\
g_n(h)
\end{bmatrix},
\end{eqnarray*}
where
\begin{equation}\label{Kn}
K_n=-{\rm i}\mu\begin{bmatrix}
0 & {\rm i}\overline{\beta_{2n}^2} & 
{\rm i}\overline{\beta_{2n}}\alpha_{2n} &
{\rm i}\overline{\beta_{1n}}\alpha_{1n}\\
-{\rm i}\overline{\beta_{2n}^2} & 0 &
-{\rm i}\overline{\beta_{2n}}\alpha_{1n} &
{\rm i}\overline{\beta_{1n}}\alpha_{2n}\\
-{\rm i}\overline{\beta_{2n}}\alpha_{2n}
&{\rm i}\overline{\beta_{2n}}\alpha_{1n} &
0 & -{\rm i}\overline{\beta_{2n}^2}
\end{bmatrix}.
\end{equation}
It follows from $\nabla\cdot\boldsymbol{q}=0$ that
\begin{eqnarray*}
q_{3n}'(x_3)={\rm i}\alpha_{1n}q_{1n}(x_3)+{\rm
i}\alpha_{2n}q_{2n}(x_3).
\end{eqnarray*}
Taking the derivative of the above equation and combining the result with
\eqref{q3Z}, we get
\begin{equation}\label{q3atx3}
{\rm i}\alpha_{1n}q_{1n}'(x_3)+{\rm
i}\alpha_{2n}q_{2n}'(x_3)=-\frac{1}{\mu}Z_{3n}
(x_3)-\left(\kappa_2^2-|\boldsymbol{\alpha}_n|^2\right)q_{3n} (x_3).
\end{equation}
Evaluating $q_{3n}(x_3)$ at $x_3=h$, we have from
\eqref{tbcgq1}--\eqref{tbcgq2} that 
\begin{eqnarray*}
-\left(\kappa_2^2-|\boldsymbol{\alpha}_n|^2\right)q_{3n}(h)=
{\rm i}\alpha_{1n}(-{\rm i}\overline{\beta_{2n}})q_{1n}(h)
+{\rm i}\alpha_{2n}(-{\rm i}\overline{\beta_{2n}})q_{2n}(h),
\end{eqnarray*}
which gives
\begin{eqnarray}\label{q3h1}
q_{3n}(h)=-\frac{\alpha_{1n}\overline{\beta_{2n}}}{\beta_{2n}^2 }q_{1n}(h)
-\frac{\alpha_{2n}\overline{\beta_{2n}}}{\beta_{2n}^2}q_{2n}(h).
\end{eqnarray}
Substituting \eqref{q3h1} into \eqref{Kn}, we obtain 
\[
D\boldsymbol{p}=T^*\boldsymbol{p}\quad {\rm
on} ~ \Gamma_{h},
\]
which completes the proof.
\end{proof}

Consider the general two-point boundary value problem for the second order
ordinary differential equation
\begin{equation*}
\begin{cases}
u''(y)-|\beta|^2 u(y)=-c\xi,\quad y\in(\hat h, h),\\
u(\hat h)=u(\hat h),\quad u'(h)=-|\beta|u(h),
\end{cases}
\end{equation*}
which has a unique solution given by 
\begin{eqnarray*}
u(y)=\frac{1}{2\left|\beta\right|}\bigg[-c\int_{h}^{y} 
e^{|\beta|(y-s)}\xi(s)\,{\rm d}s+c\int_{\hat
h}^{y}e^{|\beta|(s-y)}\xi(s)\,{\rm d}s\\
-c\int_{\hat h}^{h}e^{|\beta|(2\hat h-y-s)}\xi(s)\,{\rm d}s
+2|\beta|e^{|\beta|(\hat h-y)}u(\hat h)\bigg].
\end{eqnarray*}

\begin{lemma}\label{lemmaP}
Let $\boldsymbol{p}=(p_1, p_2, p_3)^\top$ be the solution
of \eqref{DualProblem3}. Then for sufficiently large $|n|_{\max}$,
the following estimate holds:
\begin{equation*}
|p_{jn}(h)|\lesssim
|n|_{\max}e^{|\beta_{2n}|(\hat h-h)}\sum\limits_{j=1,2,3}
|p_{jn}(\hat h)|
+\frac{1}{|n|_{\max}}\sum\limits_{j=1,2,3}
\|\xi_{jn}\|_{L^{\infty}([\hat h, h])},
\end{equation*}	
where $p_{jn}$ are the Fourier coefficients of $p_j, j=1,2,3$.
\end{lemma}

\begin{proof}
Let $c_1=1/\left(\lambda+2\mu\right) $ and $c_2=1/\mu$. We solve the two-point
boundary value problems of \eqref{g3Z}--\eqref{q3Z} and get the solutions
\begin{eqnarray}
g_n(x_3)=\frac{1}{2|\beta_{1n}|} \bigg[-c_1\int_{h}^{x_3}  
e^{|\beta_{1n}|(x_3-s)}\zeta_n(s)\,{\rm d}s +c_1\int_{\hat
h}^{x_3}e^{|\beta_{1n}|(s-x_3)}\zeta_n(s)\,{\rm d}s \notag\\
-c_1\int_{\hat h}^{h}e^{|\beta_{1n}|(2\hat h-x_3-s)}\zeta_n(s)\,{\rm d}s
+2|\beta_{1n}|e^{|\beta_{1n}|(\hat h-x_3)}g_n(\hat h)\bigg], 
\label{g}
\end{eqnarray}
\begin{eqnarray}
q_{1n}(x_3)=\frac{1}{2|\beta_{2n}|}
\bigg[-c_2\int_{h}^{x_3}e^{|\beta_{2n}|(x_3-s)}Z_{1n}(s)\,{\rm d}s
+c_2\int_{\hat h}^{x_3}e^{|\beta_{2n}|(s-x_3)}Z_{1n}(s)\,{\rm d}s \notag\\
-c_2\int_{\hat h}^{h}e^{|\beta_{2n}|(2\hat h-x_3-s)}Z_{1n}(s)\,{\rm d}s
+2|\beta_{2n}|e^{|\beta_{2n}|(\hat h-x_3)}q_{1n}(\hat h)\bigg],
\label{q1}
\end{eqnarray}
\begin{eqnarray}
q_{2n}(x_3)=\frac{1}{2|\beta_{2n}|}\bigg[-c_2\int_{h}^{x_3}  
e^{|\beta_{2n}|(x_3-s)}Z_{2n}(s)\,{\rm d}s+c_2\int_{\hat
h}^{x_3}e^{|\beta_{2n}|(s-x_3)}Z_{2n}(s)\,{\rm d}s \notag\\
-c_2\int_{\hat h}^{h}e^{|\beta_{2n}|(2\hat h-x_3-s)}Z_{2n}(s)\,{\rm d}s
+2|\beta_{2n}|e^{|\beta_{2n}|(\hat h-x_3)}q_{2n}(\hat h)\bigg].
\label{q2}
\end{eqnarray}
Taking the derivatives of \eqref{g}--\eqref{q2} and then evaluating at
$x_3=\hat h$ gives
\begin{eqnarray}
q_{1n}'(\hat h)&=&c_2\int_{\hat h}^{h}e^{|\beta_{2n}|(\hat h-s)}
Z_{1n}(s)\,{\rm d}s-|\beta_{2n}|q_{1n}(\hat h),\label{dq1h1}\\
q_{2n}'(\hat h)&=&c_2\int_{\hat h}^{h}e^{|\beta_{2n}|(\hat h-s)}
Z_{2n}(s)\,{\rm d}s-|\beta_{2n}|q_{2n}(\hat h),\label{dq2h1}\\	
g_n'(\hat h)&=&c_1\int_{\hat h}^{h}e^{|\beta_{1n}|(\hat h-s)}\zeta_n(s)\,{\rm
d}s-|\beta_{1n}|g_n(\hat h).
\label{dgh1}
\end{eqnarray}
Evaluating \eqref{q3atx3} at $x_3=\hat h$ and then using
\eqref{dq1h1}--\eqref{dq2h1}, we get
\begin{eqnarray}
&&q_{3n}(\hat h)=-\frac{{\rm i}\alpha_{1n}}{|\beta_{2n}|}q_{1n}(\hat h)
-\frac{{\rm i}\alpha_{2n}}{|\beta_{2n}|}q_{2n}(\hat h_2)
+\frac{1}{|\beta_{2n}|^2}\frac{1}{\mu}Z_{3n}(\hat h) \notag\\
&&\quad +\frac{{\rm i}\alpha_{1n}}{|\beta_{2n}|^2}
c_2\int_{\hat h}^{h}e^{|\beta_{2n}|(\hat h-s)}Z_{1n}(s)\,{\rm d}s
+\frac{{\rm i}\alpha_{2n}}{|\beta_{2n}|^2}c_2\int_{\hat
h}^{h}e^{|\beta_{2n}|(\hat h-s)}Z_{2n}(s)\,{ \rm d}s
\label{q3h2}
\end{eqnarray}

Plugging  \eqref{tbcgq1}--\eqref{tbcgq2} and \eqref{q3atx3} into
\eqref{valueP} yields 
\begin{equation}\label{pKqh1}
(p_{1n}(h), p_{2n}(h), p_{3n}(h))^\top
=\frac{1}{|\beta_{2n}|}K_n (q_{1n}(\hat h), q_{2n}(\hat h), g_n(\hat h))^\top,
\end{equation}
where
\begin{eqnarray*}
K_n=\begin{bmatrix}
\alpha_{1n}\alpha_{2n} & |\beta_{2n}|^2
+\alpha_{2n}^2 & {\rm i}\alpha_{1n}|\beta_{2n}|\\
-|\beta_{2n}|^2-\alpha_{1n}^2 & -\alpha_{1n}\alpha_{2n} 
& {\rm i}\alpha_{2n}|\beta_{2n}|\\
-{\rm i}\alpha_{2n}|\beta_{2n}| & {\rm i}\alpha_{1n}|\beta_{2n}| 
&-|\beta_{1n}| |\beta_{2n}|
\end{bmatrix}. 
\end{eqnarray*}
It follows from a straightforward calculation that the inverse of $K_n$ is
\begin{equation}\label{Kinverse}
K_n^{-1}=\frac{1}{|\beta_{2n}|^2 \chi_n\left(|\beta_{2n}|^2
+|\boldsymbol{\alpha}_n|^2\right)}\hat{K}_n,
\end{equation}
where the entries of the matrix $\hat K_n$ are 
\begin{eqnarray*}
&&
\hat{K}^{(n)}_{11}=\alpha_{1n}\alpha_{2n}|\beta_{2n}|\left(|\beta_{1n}|
+|\beta_{2n}|\right), \quad
\hat{K}^{(n)}_{13}={\rm i}\alpha_{2n}|\beta_{2n}|\left(|\beta_{2n}|^2
+|\boldsymbol{\alpha}_n|^2\right),\\
&& \hat{K}^{(n)}_{12}=-\alpha_{1n}^2|\beta_{2n}|^2
+|\beta_{1n}| |\beta_{2n}|^3+\alpha_{2n}^2 |\beta_{1n}| |\beta_{2n}|,\\
&&\hat{K}^{(n)}_{21}=\alpha_{2n}^2 |\beta_{2n}|^2
-|\beta_{1n}||\beta_{2n}|^3-\alpha_{1n}^2|\beta_{1n}| |\beta_{2n}|,\\
&& \hat{K}^{(n)}_{22}=-\alpha_{1n}\alpha_{2n}
|\beta_{2n}|\left(|\beta_{1n}| +|\beta_{2n}|\right) ,\quad
\hat{K}^{(n)}_{23}=-{\rm
i}\alpha_{1n}|\beta_{2n}|\left(|\beta_{2n}|^2
+|\boldsymbol{\alpha}_n|^2\right),\\
&&\hat{K}^{(n)}_{31}=-{\rm
i}\alpha_{1n}|\beta_{2n}|\left(|\beta_{2n}|^2
+|\boldsymbol{\alpha}_n|^2\right),\quad
\hat{K}^{(n)}_{32}=-{\rm i}\alpha_{2n}|\beta_{2n}|\left(|\beta_{2n}|^2
+|\boldsymbol{\alpha}_n|^2\right),\\
&&\hat{K}^{(n)}_{33}=|\beta_{2n}|^2\left(|\beta_{2n}|^2
+|\boldsymbol{\alpha}_n|^2\right).
\end{eqnarray*}

Evaluating  \eqref{g}--\eqref{q2} at $x_3=h$, we get
\begin{eqnarray}\label{pq3}
\begin{bmatrix}
q_{1n}(h)\\
q_{2n}(h)\\
g_n(h)
\end{bmatrix}=
\begin{bmatrix}
e^{|\beta_{2n}|(\hat h-h)} & 0 & 0\\
0 & e^{|\beta_{2n}|(\hat h-h)} & 0\\
0 & 0 & e^{|\beta_{1n}|(\hat h-h)}
\end{bmatrix}
\begin{bmatrix}
q_{1n}(\hat h)\\
q_{2n}(\hat h)\\
g_n(\hat h)
\end{bmatrix}
+\begin{bmatrix}
\hat{w}_{1n} \\ 
\hat{w}_{2n} \\
\hat{w}_{3n} 
\end{bmatrix},
\end{eqnarray}
where
\begin{eqnarray*}
&&\hat{w}_{1n}=\frac{c_2}{2|\beta_{2n}|}\left[
\int_{\hat h}^{h}e^{|\beta_{2n}|(s-h)}Z_{1n}(s)\,{\rm d}s
-\int_{\hat h}^{h}e^{|\beta_{2n}|(2\hat h-h-s)}Z_{1n}(s)\,{\rm d}s\right],\\
&&\hat{w}_{2n}=\frac{c_2}{2|\beta_{2n}|}\left[
\int_{\hat h}^{h}e^{|\beta_{2n}|(s-h)}Z_{2n}(s)\,{\rm d}s
-\int_{\hat h}^{h}e^{|\beta_{2n}|(2\hat h-h-s)}Z_{2n}(s)\,{\rm d}s\right],\\
&&\hat{w}_{3n}=\frac{c_1}{2|\beta_{1n}|}\left[
\int_{\hat h}^{h}e^{|\beta_{1n}|(s-h)}\zeta_n(s)\,{\rm d}s
-\int_{\hat h}^{h}e^{|\beta_{1n}|(2\hat h-h-s)}
\zeta_n(s)\,{\rm d}s\right].
\end{eqnarray*}
Similarly, we evaluate \eqref{valueP} at $x_3=\hat h$ and get 
\begin{eqnarray}\label{pq1}
\begin{bmatrix}
p_{1n}(\hat h)\\
p_{2n}(\hat h)\\
p_{3n}(\hat h) \end{bmatrix}
=\frac{1}{|\beta_{2n}|}K_n
\begin{bmatrix}
q_{1n}(\hat h)\\
q_{2n}(\hat h)\\
g_n(\hat h) \end{bmatrix}
+\begin{bmatrix}
w_{1n}\\
w_{2n}\\
w_{3n} 
\end{bmatrix},
\end{eqnarray}
where
\begin{eqnarray*}
w_{1n}
&=&-\frac{1}{|\beta_{2n}|^2}\bigg(c_2\alpha_{1n}
\alpha_{2n}\int_{\hat h}^{h}e^{|\beta_{2n}|(\hat h-s)}Z_{1n}(s)\,{\rm d}s
-{\rm i}\alpha_{2n}\frac{1}{\mu}Z_{3n}(\hat h)\\
&&+c_2\left(|\beta_{2n}|^2+\alpha_{2n}^2\right)\int_{\hat h}^{h}
e^{|\beta_{2n}|(\hat h-s)}Z_{2n}(s)\,{\rm d}s\bigg),\\
w_{2n} &=&\frac{1}{|\beta_{2n}|^2}\bigg(c_2\alpha_{1n}
\alpha_{2n}\int_{\hat h}^{h} e^{|\beta_{2n}|(\hat h-s)}Z_{2n}(s)\,{\rm d}s-{\rm
i}\alpha_{1n}\frac{1}{\mu}Z_{3n}(\hat h)\\
&&+c_2\left(|\beta_{2n}|^2+\alpha_{1n}^2\right)\int_{\hat h}^{h}
e^{|\beta_{2n}|(\hat h-s)}Z_{1n}(s)\,{\rm d}s\bigg),\\
w_{3n}&=&c_1\int_{\hat h}^{h}e^{|\beta_{1n}|(\hat h-s)}\zeta_n(s)\, {\rm d}s. 
\end{eqnarray*}
It follows from \eqref{Kinverse} that we have 
\begin{equation}\label{pq2}
\begin{bmatrix}
q_{1n}(\hat h) \\ 
q_{2n}(\hat h) \\
g_n(\hat h)
\end{bmatrix}=
|\beta_{2n}|K_n^{-1}
\begin{bmatrix}
p_{1n}(\hat h) \\ 
p_{2n}(\hat h) \\ 
p_{3n}(\hat h)\end{bmatrix}
-|\beta_{2n}|K_n^{-1}
\begin{bmatrix}
w_{1n} \\ 
w_{2n} \\ 
w_{3n}
\end{bmatrix}.
\end{equation}
Substituting \eqref{pq3} into \eqref{pKqh1} leads to 
\begin{eqnarray*}
&&\begin{bmatrix}
p_{1n}(h) \\ 
p_{2n}(h) \\
p_{3n}(h)
\end{bmatrix}=
\frac{1}{|\beta_{2n}|}K_n
\begin{bmatrix}
q_{1n}(h) \\ 
q_{2n}(h) \\ 
g_n(h)
\end{bmatrix}\\
&&=\frac{1}{|\beta_{2n}|}
K_n\begin{bmatrix}
e^{|\beta_{2n}|(\hat h-h)} & 0 & 0\\
0 & e^{|\beta_{2n}|(\hat h-h)} & 0\\
0 & 0 & e^{|\beta_{1n}|(\hat h-h)}
\end{bmatrix}\begin{bmatrix}
q_{1n}(\hat h) \\
q_{2n}(\hat h) \\ 
g_n(\hat h)
\end{bmatrix}
+\frac{1}{|\beta_{2n}|}K_n
\begin{bmatrix}
\hat{w}_{1n} \\ 
\hat{w}_{2n} \\ 
\hat{w}_{3n}
\end{bmatrix}.
\end{eqnarray*}
Plugging \eqref{pq2} into the above equation gives 
\begin{equation}\label{ph1ph2}
\begin{bmatrix}
p_{1n}(h) \\ 
p_{2n}(h) \\
p_{3n}(h)
\end{bmatrix}=
P_n
\begin{bmatrix}
p_{1n}(\hat h) \\
p_{2n}(\hat h) \\ 
p_{3n}(\hat h)
\end{bmatrix}
-P_n
\begin{bmatrix}
w_{1n} \\
w_{2n} \\
w_{3n} 
\end{bmatrix}+\frac{1}{|\beta_{2n}|}
K_n\begin{bmatrix}
\hat{w}_{1n} \\
\hat{w}_{2n} \\ 
\hat{w}_{3n}
\end{bmatrix},
\end{equation}
where the matrix $P_n$ is given in \eqref{decayU}. 

Following the same proof as that for \cite[Lemma 5.8]{LY-2019-periodic}, we
may show that 
\begin{equation}\label{wnestimate}
|\hat{w}_{jn}|\lesssim
\frac{1}{|n|_{\max}^2}\sum\limits_{j=1,2,3} 
\|\xi_{jn}\|_{L^{\infty}([\hat h, h])},
\quad |w_{jn}|\lesssim
\frac{1}{|n|_{\max}^2}\sum\limits_{j=1,2,3} 
\|\xi_{jn}\|_{L^{\infty}([\hat h, h])}. 
\end{equation}
By \eqref{aymptoticP} and \eqref{wnestimate}, we have
\begin{equation}\label{Pw}
\bigg|\sum\limits_{j=1,2,3}P^{(n)}_{ij}
w_{jn}\bigg|\lesssim \frac{1}{|n|_{\max}}
e^{\left(|\boldsymbol{\alpha}_n|
-|\beta_{2n}|\right)(h-\hat h)}\sum\limits_{j=1,2,3}
\|\xi_{jn}\|_{L^{\infty}([\hat h, h])}.	
\end{equation}
For sufficiently large $|n|_{\max}$, it is easy to get 
\begin{eqnarray*}
|\boldsymbol{\alpha}_n|-|\beta_{2n}|=|\boldsymbol{\alpha}_n|
-(|\boldsymbol{\alpha}_n|^2-\kappa_2^2)^{1/2}=\frac{\kappa_2^2}{
|\boldsymbol{\alpha}_n|+(|\boldsymbol{\alpha}_n|^2
-\kappa_2^2)^{1/2}}\sim\frac{1}{|n|_{\max}}.
\end{eqnarray*}
Plugging the above estimate into \eqref{Pw} gives 
\begin{equation}\label{Pwestimate}
\bigg|\sum\limits_{j=1,2,3}P^{(n)}_{ij}
w_{jn}\bigg|\lesssim\frac{1}{|n|_{\max}}\sum\limits_{j=1,2,3}
\|\xi_{jn}\|_{L^{\infty}([\hat h, h])}.
\end{equation}
It is also easy to check
\[
\left|\frac{1}{|\beta_{2n}|}K_n
\right|\sim O(|n|_{\max}).
\]
By \eqref{wnestimate}, we have 
\begin{equation}\label{Khatw}
\frac{1}{|\beta_{2n}|}\bigg|\sum\limits_{j=1,2,3}K^{(n)}_{ij}
\hat{w}_{jn}\bigg|\lesssim\frac{1}{|n|_{\max}}
\sum\limits_{j=1,2,3} \|\xi_{jn}\|_{L^{\infty}([\hat h, h])},
\end{equation}
which completes the proof after substituting \eqref{aymptoticP} and 
\eqref{Pwestimate}--\eqref{Khatw} into \eqref{ph1ph2}.
\end{proof}

Using Lemma \ref{lemmaP} and the same arguments as those
in \cite{LY-2019-periodic}, we may show that 
\begin{equation}\label{xiE2}
\bigg|\int_{\Gamma_{h}}(T-T_N)\boldsymbol{\xi}
\cdot\overline{\boldsymbol{p}}\,{\rm d}s\bigg|
\lesssim \frac{1}{N}\|\boldsymbol{\xi}\|^2_{\boldsymbol{H}^1(\Omega)}.
\end{equation}
The details are omitted for brevity. 

Now we are ready to show the proof of Theorem \ref{mainthm}.

\begin{proof}
By \eqref{Error}, Lemmas \ref{lemma3} and \ref{lemma5}, we obtain 
\begin{eqnarray*}
\vvvert \boldsymbol{\xi}\vvvert^2_{\boldsymbol{H}^{1}(\Omega)}\leq
C_1\left[\bigg(\sum\limits_{K\in
\mathcal M_h}\eta_{K}^2\bigg)^{1/2}+\max_{|n|_{\min}>N}
\Big(|n|_{\max}e^{-|\beta_{2n}|(h-\hat h)}
\Big)\|\boldsymbol{u}^{\rm inc}\|_{\boldsymbol{H}^{1}(\Omega)}
\right]\|\boldsymbol{\xi}\|_{\boldsymbol{H}^{1}(\Omega)}\\
+\left(C_2+C(\delta)\right)\|\boldsymbol{\xi}\|^2_{\boldsymbol{L}
^2(\Omega)}+\delta\|\boldsymbol{\xi}\|^2_{\boldsymbol{H}^{1}(\Omega)},
\end{eqnarray*}
where $C_1, C_2, C(\delta)$ are positive constants. Choosing a small enough
$\delta$ such that $\delta/\min(\mu, \omega^2)<1/2$ gives 
\begin{eqnarray}
\vvvert \boldsymbol{\xi}\vvvert_{\boldsymbol{H}^1(\Omega)}^2\leq 2C_1
\left[\bigg(\sum\limits_{K\in
\mathcal M_h}\eta_{K}^2\bigg)^{1/2}+\max_{|n|_{\min}>N}
\Big(|n|_{\rm max}e^{-|\beta_{2n}|(h-\hat h)}
\Big)\|\boldsymbol{u}^{\rm inc}\|_{\boldsymbol{H}^{1}(\Omega)}
\right]\|\boldsymbol{\xi}\|_{\boldsymbol{H}^{1}(\Omega)} \notag\\
+2(C_2+C(\delta))\|\boldsymbol{\xi}\|^2_{\boldsymbol{L}^2(\Omega)}.
\label{xiH1}
\end{eqnarray}
Substituting  \eqref{xiE2} into \eqref{xiL2} and using Lemma \ref{lemma3}, we
have 
\begin{eqnarray}\label{xiL22}
\|\boldsymbol{\xi}\|^2_{\boldsymbol{L}^2(\Omega)}\lesssim \left[
\bigg(\sum\limits_{K\in \mathcal M_h}\eta_{K}^2\bigg)^{1/2}+\max_{|n|_{\min}>N}
\Big(|n|_{\max}e^{-|\beta_{2n}|(h-\hat h)}
\Big)\|\boldsymbol{u}^{\rm inc}\|_{\boldsymbol{H}^{1}(\Omega)}
\right]\|\boldsymbol{\xi}\|_{\boldsymbol{H}^{1}(\Omega)}\notag\\
+\frac{1}{N}\|\boldsymbol{\xi}\|^2_{\boldsymbol{H}^1(\Omega)}.
\end{eqnarray}
The proof is completed after substituting \eqref{xiL22} into \eqref{xiH1} and
taking $N$ be a sufficiently large number. 
\end{proof}

\section{Numerical experiments}\label{section: NE}

In this section, we introduce the algorithmic implementation of the adaptive
finite element DtN method and present two numerical examples to demonstrate the
effectiveness of the proposed method. 

\subsection{Adaptive algorithm}

It is shown in Theorem \ref{mainthm} that the a posteriori error consists of two
parts: the finite element discretization error $\epsilon_h$ and the DtN operator
truncation error $\epsilon_N$, where 
\begin{eqnarray}\label{epsilonN}
\epsilon_h = \left(\sum\limits_{K\in\mathcal M_h} \eta^2_{K}\right)^{1/2},\quad 
\epsilon_N =\max\limits_{|n|_{\min}>N}
\left(|n|_{\rm max}e^{-|\beta_{2n}|(h-\hat h)}\right)\|\boldsymbol{u}^{\rm
inc}\|_{\boldsymbol{H}^{1}(\Omega)}. 
\end{eqnarray}
In the implementation, we choose the parameters $h, \hat h$ and $N$ based on
\eqref{epsilonN} to make sure that the DtN operator truncation error is smaller
than the finite element discretization error. 
In the following numerical experiments, $\hat h$ is chosen such that
$\hat h=\max_{\boldsymbol{r}\in\mathbb{R}^2}f(\boldsymbol{r})$ and $N$
is the smallest positive integer that makes $\epsilon_N\leq 10^{-8}$. 
The adaptive finite element DtN algorithm is shown in Table 1.

\begin{table}\label{Table1}	
\caption{The adaptive finite element DtN method.}
\vspace{1ex}
\hrule \hrule
\vspace{0.8ex} 
\begin{enumerate}		
\item Given the tolerance $\epsilon>0$ and the parameter $\tau\in(0,1)$.
\item Fix the computational domain $\Omega$ by choosing $h$.
\item Choose $\hat h$ and $N$ such that $\epsilon_N\leq 10^{-8}$.
\item Construct an initial triangulation $\mathcal M_h$ over $\Omega$ and compute
error estimators.
\item While $\epsilon_h>\epsilon$ do
\item \qquad refine mesh $\mathcal M_h$ according to the following strategy: \\
\[\text{if } \eta_{\hat{K}}>\tau \max\limits_{K\in\mathcal M_h}
\eta_{K}, \text{ refine the element } \hat{K}\in\mathcal M_h, \]
\item \qquad denote refined mesh still by $\mathcal M_h$, solve the discrete
problem \eqref{variation2} on the new mesh $\mathcal M_h$,
\item \qquad compute the corresponding error estimators.
\item End while.
\end{enumerate}
\vspace{0.8ex}
\hrule\hrule
\vspace{0.8ex} 
\end{table}

\subsection{Numerical examples}

In this section, we present two examples (cf. \cite{JLLZ-cms18}) to demonstrate
the numerical performance of the DtN method. The first-order linear element is
used for solving the problem. Our implementation is based on parallel
hierarchical grid (PHG) \cite{phg}, which is a toolbox for developing parallel
adaptive finite element programs on unstructured tetrahedral meshes. The linear
system resulted from the finite element discretization is solved by the
Supernodal LU (SuperLU) direct solver, which is a general purpose library for
the direct solution of large, sparse, nonsymmetric systems of linear equations.

{\em Example 1}. Consider a simple biperiodic structure, a plane surface,
where the exact solution is available. We assume that a plane compressional
plane wave $\boldsymbol u^{\rm inc}=\boldsymbol q e^{{\rm
i}(\boldsymbol\alpha\cdot \boldsymbol r-\beta x_3)}$ is incident on the
plane surface $x_3=0$, where $\boldsymbol\alpha=(\alpha_1, \alpha_2)^\top,
\alpha_1=\kappa_1\sin\theta_1\cos\theta_2,
\alpha_2=\kappa_1\sin\theta_1\sin\theta_2, \beta=\kappa_1\cos\theta_1,
\boldsymbol q=(q_1, q_2, q_3)^\top, q_1=\sin\theta_1\cos\theta_2,
q_2=\sin\theta_1\sin\theta_2, q_3=-\cos\theta_1, \theta_1\in [0, \pi/2),
\theta_2\in[0, 2\pi]$ are incident angles. It follows from the elastic wave
equation and the Helmholtz decomposition that we may obtain the exact solution
for the scattered field 
\[
\boldsymbol u(\boldsymbol x)={\rm i}
\begin{bmatrix}
\alpha_1\\
\alpha_2\\
\beta
\end{bmatrix}
 a e^{{\rm i}(\boldsymbol\alpha\cdot\boldsymbol r+\beta x_3)}+{\rm i}
\begin{bmatrix}
\alpha_2 b_3 -
\beta_{2 0}b_2\\
\beta_{2 0}b_1-\alpha_1 b_3\\
\alpha_1 b_2-\alpha_2 b_1
\end{bmatrix}
 e^{{\rm i}(\boldsymbol\alpha\cdot\boldsymbol r+\beta_{20}x_3)},
\]
where $(a, b_1, b_2, b_3)^\top$ is the solution of the following linear system:
\[
 {\rm i}
 \begin{bmatrix}
 \alpha_{1} &                        0 & -\beta_{20} &
\alpha_{2}\\
 \alpha_{2} & \beta_{20} &         0 & -\alpha_{1}\\
 \beta &    -\alpha_{2} &  \alpha_{1} &
0\\
 0 &          \alpha_{1} &  \alpha_{2} & \beta_{20}
 \end{bmatrix}
 \begin{bmatrix}
  a\\
  b_1\\
  b_2\\
  b_3
 \end{bmatrix}=-
 \begin{bmatrix}
  q_1\\
  q_2\\
  q_3\\
  0
 \end{bmatrix}.
 \]
 Solving the above equations via Cramer's rule gives
  \begin{eqnarray*}
a&=&\frac{\rm i}{\chi}\big(\alpha_{1}   q_1
+ \alpha_{2}  q_2+ \beta_{20} q_3\big),\\
b_{1}  &=&\frac{\rm i}{\chi}\big(
 \alpha_{1}\alpha_{2}(\beta-\beta_{20}) q_1 /\kappa^2_{2}
 +(\alpha_{1}^2\beta_{20}+\alpha_{2}^2\beta +
\beta \beta_{20}^2) q_2/\kappa^2_{2}
 -\alpha_{2} q_3\big), \\
b_{2}  &=&\frac{\rm i}{\chi}\big(
 -(\alpha_{1}^2\beta+\alpha_{2}^2\beta_{20} +
\beta \beta_{20}^2) q_1/\kappa^2_{2}
  -\alpha_{1}\alpha_{2}(\beta-\beta_{20})q_2/\kappa^2_{2} +\alpha_{1}
q_3\big), \\
b_{3}  &=&\frac{\rm i}{\kappa^2_{2} }\big(
 \alpha_{2} q_1 -\alpha_{1} q_2\big),
 \end{eqnarray*}
 where
 \[
 \chi=|\boldsymbol{\alpha}|^2+\beta\beta_{20}.
 \]
In our experiments, the parameters are chosen as $\lambda=1,
\mu=1, \theta_1=\theta_2=\pi/6, \omega=2\pi$. The computational domain
$\Omega=(0,1)\times(0,1)\times(0,0.3)$. The mesh and surface plots of the
amplitude of the scattered field $\boldsymbol{v}_h$ are shown in Figure
\ref{ex1:mesh}. The mesh has 228400 tetrahedrons and the total number of degrees
of freedom (DoFs) on the mesh is 253200. The grating efficiencies are displayed
in Figure \ref{ex1:eff}, which verifies the conservation of the energy in
\cite[Theorem 2.1]{JLLZ-cms18}. Figure
\ref{ex1:err} shows the curves of $\log \|\nabla(\boldsymbol{u}
-\boldsymbol{u}_k)\|_{0, \Omega}$ versus $\log N_k$  
for both the a priori and the a posteriori error estimates,
where $N_k$ is the total number of DoFs of the mesh. It
indicates that the meshes and the associated numerical complexity are
quasi-optimal, i.e.,  $\log \|\nabla(\boldsymbol{u}
-\boldsymbol{u}_k)\|_{0, \Omega}=O(N^{-1/3}_k)$ is valid
asymptotically.

{\em Example 2}. This example concerns the scattering of a time-harmonic
compressional plane wave $\boldsymbol{u}^{\rm inc}$ on a flat grating surface
with two square bumps, as seen in Figure \ref{ex2:geo}. The parameters are
chosen as $\lambda=1, \mu=1, \theta_1=\theta_2=\pi/6$, $\omega=2\pi$. The
computational domain is $\Omega=(0,1)\times(0,1)\times(0,0.6)$.
Since there is no exact solution for this example, we plot in Figure
\ref{ex2:est} the curves of $\log \|\nabla(\boldsymbol{u}
-\boldsymbol{u}_k)\|_{0, \Omega}$ versus $\log N_k$ for the a posteriori error
estimates, where $N_k$  is the total number of DoFs of the mesh. Again, the
result shows that the meshes and the associated numerical complexity are
quasi-optimal for the proposed method. We also plot the grating efficiencies
against the DoFs in Figure \ref{ex2:eff} to verify the conservation of the
energy. Figures \ref{ex2:mesh} and \ref{ex2:mesh2} show the meshes and the
amplitude of the associated solution for the scattered field $\boldsymbol
u_h $ when the mesh has 346734 tetrahedrons.

\begin{figure}
\centering
\includegraphics[width=0.4\textwidth]{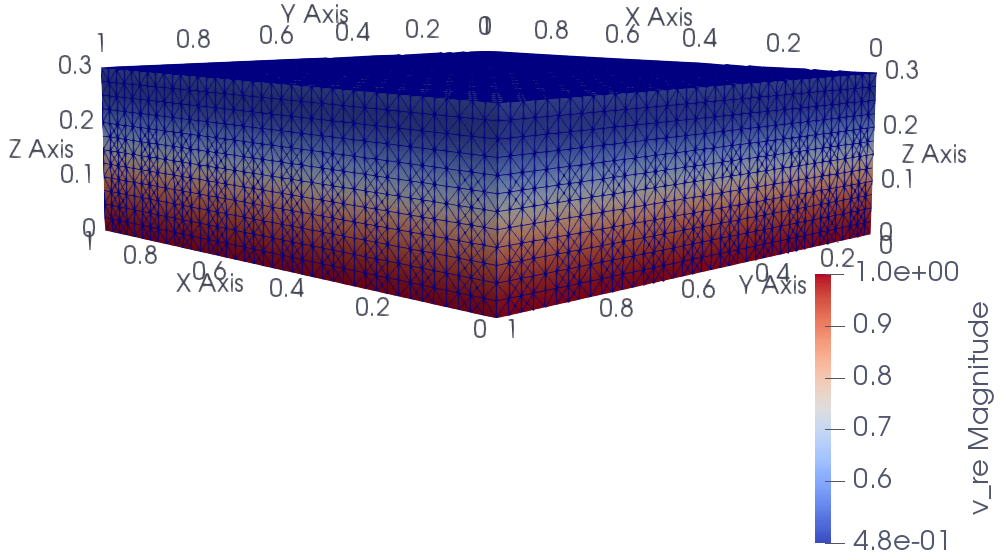}
\hskip 1.5cm
\includegraphics[width=0.4\textwidth]{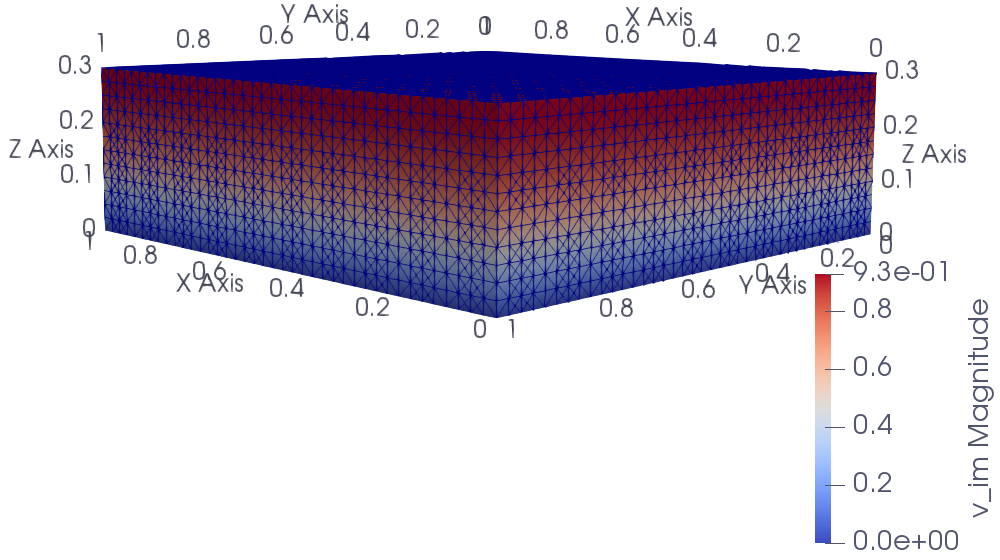}
\caption{Example 1. The mesh and surface plots of the amplitude of the
associated solution for the scattered field $\boldsymbol u_h $. (left) The
amplitude of the real part of the solution
$|\Re\boldsymbol u_h |$; (right) The amplitude of the imaginary
part of the solution $|\Im\boldsymbol u_h |$.}\label{ex1:mesh}
\end{figure}

\begin{figure}
\centering
\includegraphics[width=0.4\textwidth]{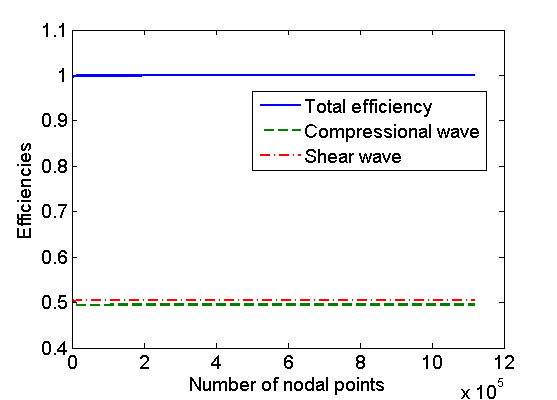}
\hskip 1.5cm
\includegraphics[width=0.4\textwidth]{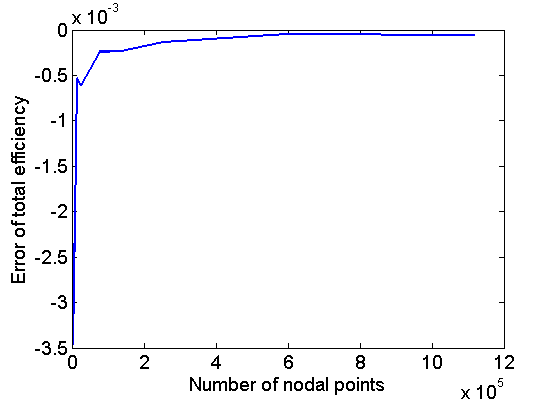}
\caption{Example 1. (left) Grating efficiencies; (right) Error of the grating
efficiency.}\label{ex1:eff}
\end{figure}

\begin{figure}
\centering
\includegraphics[width=0.42\textwidth]{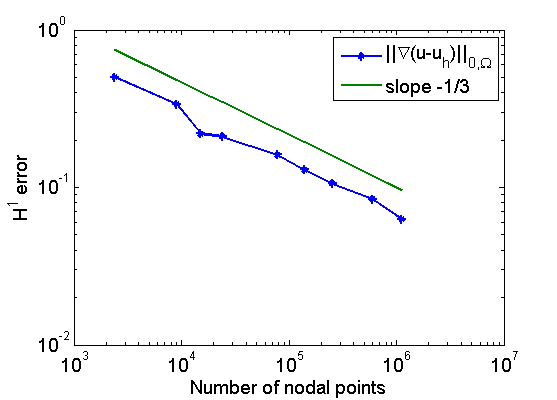}
\hskip 1.5cm
\includegraphics[width=0.42\textwidth]{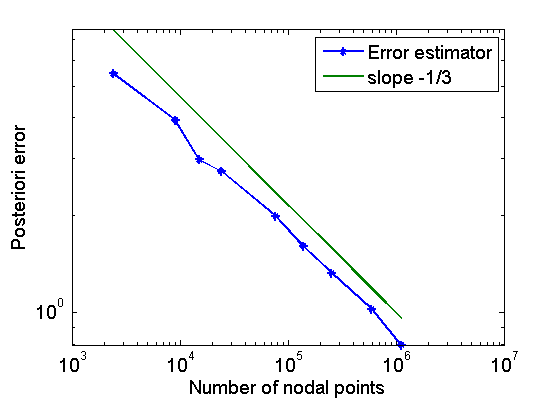}
\caption{Example 1. (left) Quasi-optimality of the a priori error estimates;
(right) Quasi-optimality of the a posteriori error estimates.}\label{ex1:err}
\end{figure}

\begin{figure}
\center
\includegraphics[width=0.4\textwidth]{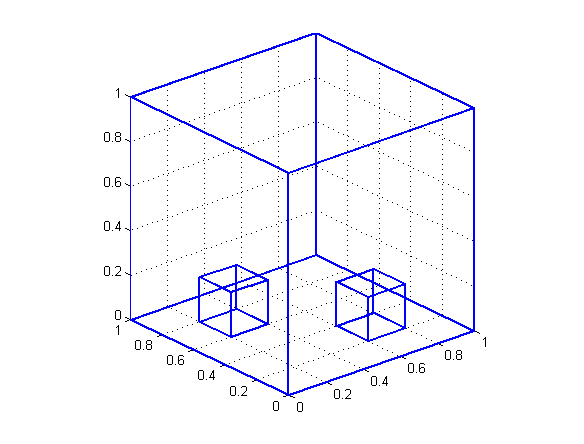}
\caption{Example 2. Problem geometry of the domain.}
\label{ex2:geo}
\end{figure}

\begin{figure}
\centering
\includegraphics[width=0.4\textwidth]{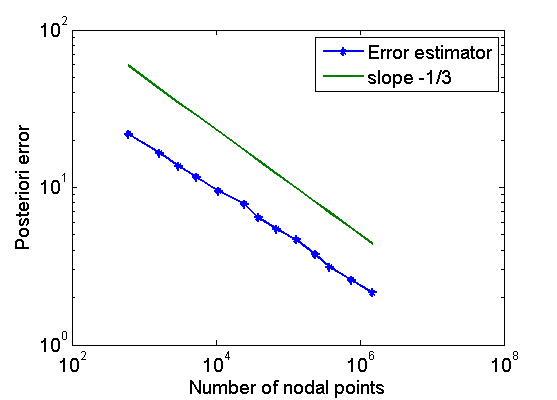}
\caption{Example 2: Quasi-optimality of the a posteriori error
estimates.}\label{ex2:est}
\end{figure}

\begin{figure}
\centering
\includegraphics[width=0.4\textwidth]{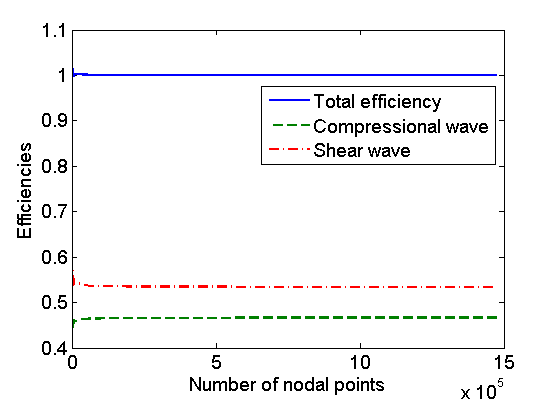}
\hskip 1.5cm
\includegraphics[width=0.4\textwidth]{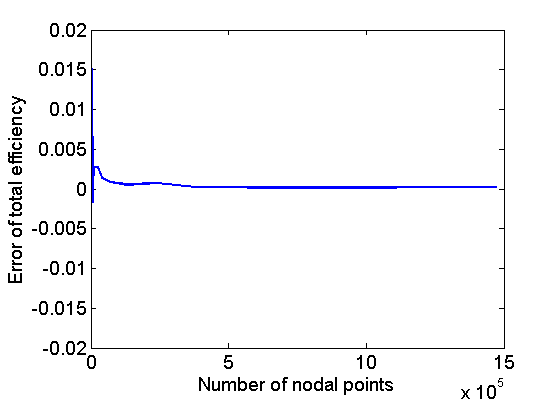}
\caption{Example 2:  Grating efficiencies; (right) Error of the grating
efficiency.}\label{ex2:eff}
\end{figure}

\begin{figure}
\centering
\includegraphics[width=0.38\textwidth]{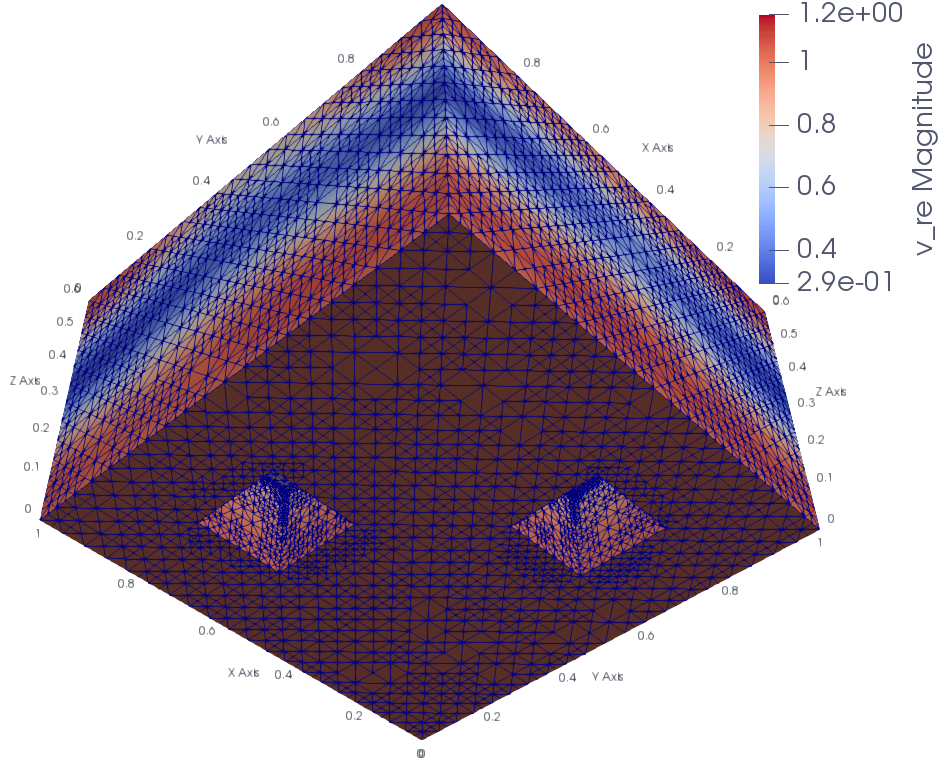}
\hskip 1.5cm
\includegraphics[width=0.38\textwidth]{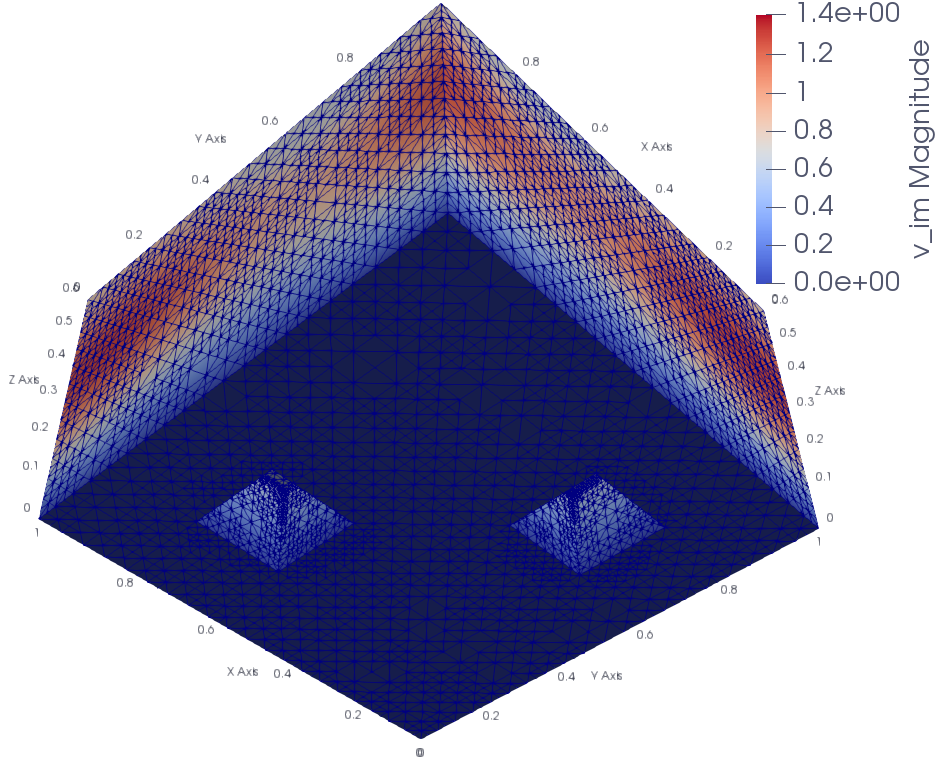}
\caption{Example 2. The mesh and surface plots of the amplitude of the
associated solution for the scattered field $\boldsymbol u_h $:
(left) the amplitude of the real part of the solution $|\Re\boldsymbol u_h |$;
(right) the amplitude of the imaginary part of the solution $|\Im\boldsymbol
u_h |$.}\label{ex2:mesh}
\end{figure}

\begin{figure}
\centering
\includegraphics[width=0.38\textwidth]{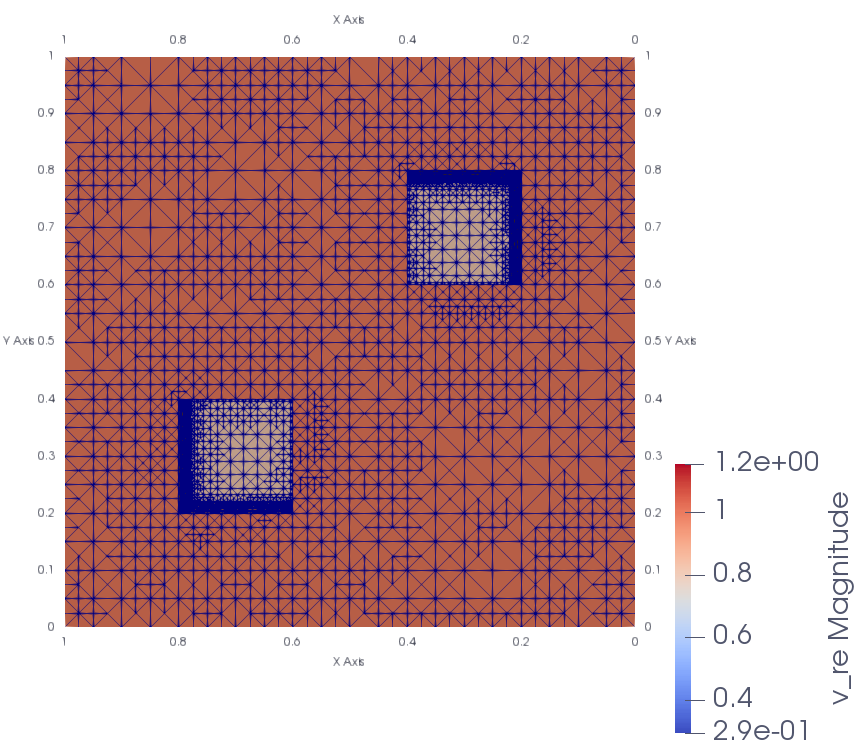}
\hskip 1.5cm
\includegraphics[width=0.38\textwidth]{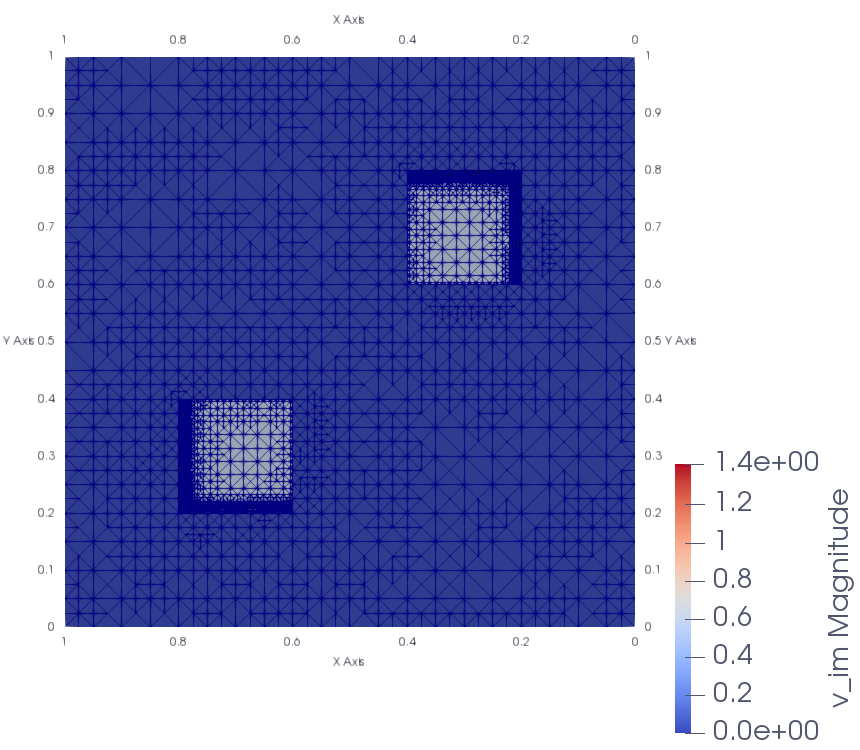}
\caption{Example 2. The mesh and surface plots of the amplitude of the
associated solution for the scattered field $\boldsymbol v_h $ from a view of
the $x_3$-axis: (left) the amplitude of the real part of the solution
$|\Re\boldsymbol v_h |$; (right) the amplitude of the imaginary part of the
solution $|\Im\boldsymbol v_h |$.}\label{ex2:mesh2}
\end{figure}

\section{conclusion}\label{section: conclusion}

In this paper, we have presented an adaptive finite element DtN method for the
elastic scattering problem in bi-periodic structures. Based on the Helmholtz
decomposition, a new duality argument is developed to obtain the a posteriori
error estimate. It takes account of both the finite element discretization error
and the DtN operator truncation error, which is shown to decay exponentially
with respect to the truncation parameter. Numerical results show that the
proposed method is effective and accurate. This work provides a viable
alternative to the adaptive finite element PML method for solving the elastic
surface scattering problem. It also enriches the range of choices available for
solving elastic wave propagation problems imposed in unbounded domains. Along
the line of this work, a possible continuation is to extend our analysis to
the adaptive finite element DtN method for solving the three-dimensional
obstacle scattering problem and acoustic-elastic interactive problem. The
progress will be reported elsewhere on these problems in the future.

\end{document}